\documentclass[12pt]{article}
\usepackage{graphicx}
\usepackage{epstopdf}
\usepackage{subfigure}
\usepackage{color}
\usepackage[svgnames]{xcolor}
\usepackage[english]{babel}
\usepackage{amsmath,amsthm}
\usepackage{array}
\usepackage{multirow}
\usepackage{amssymb}
\usepackage{isomath}
\usepackage{blkarray}
\usepackage{cancel}
\usepackage{extarrows}

\usepackage{amsfonts}
\usepackage{booktabs}
\usepackage{algorithm}
\usepackage{algorithmic}
\usepackage{pifont}
\usepackage{bm} 
\usepackage{authblk}
\usepackage{appendix}
\usepackage{mathtools}
\usepackage{float} 

\usepackage{caption} % 关键宏包
\usepackage{graphicx}
\usepackage{subfigure}
\usepackage{tabularx}
\usepackage[colorlinks=true, linkcolor=red, citecolor=green]{hyperref}
\usepackage{url}
\newtheorem{thm}{Theorem}[section]
\newtheorem{defi}{Definition}[section]
\newtheorem{lem}[thm]{Lemma}

\newtheorem{coro}{Corollary}[section]

\newtheorem{Remark}{Remark}[section]

\newtheorem{Problem}{Problem}[section]

\title{Partial Petrial Polynomials of Ribbon Graphs}

\author{Xiaoxiang Yu$^{a}$, Rong-Xia Hao$^{a }$\thanks{Corresponding author}, Jianbing Liu$^{a }$, Zhiguo Li$^{b}$,\\
	\small $^{a}$School of Mathematics and Statistics,\\
	\small Beijing Jiaotong University, Beijing 100044, P. R. China\\
	\small $^b$School of Science,\\
	\small Hebei University of Technology, Tianjin 300401, P. R. China\\
	\small Email: yuxiaoxiang1126@163.com, rxhao@bjtu.edu.cn,\\ \small jbliu1@bjtu.edu.cn, zhiguolee@hebut.edu.cn}
\date{}

\begin{document}
\baselineskip 0.65cm
\maketitle

\vspace{-2.6em}
\begin{abstract}
Gross, Mansour, and Tucker [European J. Combin., 95 (2021): 103329] introduced the \emph{partial Petrial polynomial} of a ribbon graph $G$, denoted by $^{\partial}{\varepsilon^{\times}_{G}}(z)$. Beck and Mellor proved, in both orientable and non-orientable cases  respectively, that the Euler genus of a bouquet equals the rank of a certain matrix over $\mathbb{GF}(2)$.
%Recently, the authors gave an equivalent representation of the partial Petrial polynomial of bouquets (ribbon graphs with one vertex) using the matrix rank. 
In this paper, we first generalize Beck and Mellor's results from bouquets to all ribbon graphs. Secondly, 
we give an equivalent representation of the partial Petrial polynomial for %from bouquets to 
all ribbon graphs. Specifically, the partial Petrial polynomial of a ribbon graph $G$ with $n$ vertices is equal to the sum of this polynomial for $2^{n-1}$ distinct bouquets. Moreover, we give the definition of a modified partial Petrial polynomial by assigning coefficients $+1$ or $-1$ to the terms in the partial Petrial polynomial such that the resulting polynomial satisfies the four-term relation for graphs. Finally, we generalize the modified partial Petrial polynomial from bouquets to all signed simple graphs and prove that this polynomial is $4$-invariant, which provides an answer to the problem posed by Lando [J.~Combin.~Theory Ser.~B,~80~(1) (2000): 104-121]: Which of the known graph invariants are $4$-invariants?
\vspace{3mm}


\end{abstract}

\section{Introduction}
A {\em ribbon graph} $G$ is a surface with boundary represented as the union of two sets of closed topological discs, called vertex-discs $V(G)$ and edge-ribbons $E(G)$, satisfying three  conditions \cite{2002_Ribbon graph}:
\begin{itemize}	
\vspace{-0.5em}
\item [(1)] the vertex-discs and edge-ribbons intersect in disjoint line segments;
\vspace{-0.5em}
\item [(2)] each such line segment lies on the boundary of precisely one vertex-disc and one edge-ribbon;
\vspace{-2em}
\item [(3)] every edge-ribbon contains exactly two such line segments.
\end{itemize}

\vspace{-0.5em}
\noindent This definition equivalently represents a graph cellularly embedded in a surface. %This definition provides an equivalent representation of a graph cellularly embedded in a surface. 
%Throughout 
In this paper, all ribbon graphs considered are connected multigraphs (i.e., may contain loops and parallel edges).\\
\indent The \textit{Petrial} operation  was first introduced by Wilson in 1979 \cite{Introduction_Petrial}. Specifically, the Petrial of a ribbon graph $G$, denoted $G^{\times}$, is obtained by detaching one end of each edge-ribbon from its incident vertex-disc, applying a twist to the edge-ribbon, and reattaching it to the vertex-disc. When the Petrial operation is applied only to a subset $A \subseteq E(G)$, the resulting ribbon graph is called the \textit{partial Petrial} of $G$ with respect to $A$, and is denoted by $G^{\times|A}$. In 2021, Gross, Mansour, and Tucker \cite{2021_II} introduced the \textit{partial Petrial polynomial}, which enumerates all partial Petrials of a ribbon graph according to their Euler genera $\varepsilon(G^{\times|A})$. It is defined as follows.

\vspace{-0.3em}
\begin{defi}\emph{(\cite{2021_II})}\label{poly_Petrial}
The partial Petrial polynomial of a ribbon graph $G$ is defined as
\vspace{-0.3em}
\[^{\partial}{\varepsilon^{\times}_{G}}(z)=\sum\limits_{A\subseteq E(G)}z^{\varepsilon{(G^{\times|A})}},\]

\vspace{-0.3em}
\noindent where the sum enumerates partial Petrials of $G$ by Euler genus.
%The partial Petrial polynomial of any ribbon graph $G$ is the generating~function
%$$^{\partial}{\varepsilon^{\times}_{G}}(z)=\sum\limits_{A\subseteq E(G)}z^{\varepsilon{(G^{\times|A})}}$$
%\noindent that enumerates partial Petrials of $G$ by Euler genus. 
\end{defi}

\vspace{-0.6em}
\noindent In~\cite{2021_II},~it was shown that the polynomial~$^{\partial}{\varepsilon^{\times}_{G}}(z)$ is \emph{interpolating},~meaning that its nonzero coefficients correspond to consecutive powers of the variable.~They also introduced the restricted orientable partial Petrial polynomial of $G$ by restricting the generating function to $A\subseteq E(G)$ such that $G^{\times|A}$ is orientable and asked whether the restricted-orientable partial Petrial polynomial of any ribbon graph is even-interpolating. Later, Chen et~al.~\cite{Counterexample_ChenYichao}~gave a negative answer.\\
\indent The {\em join} $G_1\vee G_2$ of two ribbon graphs $G_1$ and $G_2$ is obtained by selecting an arc $p_i$ on the boundary of a vertex-disc in each $G_i$ for $i=1,2$ between two consecutive ribbon ends and identifying  $p_1$ and $p_2$ to merge the two vertex-discs into one \cite{Join_Moffatt}.~A ribbon graph $G$ is called {\em prime} if there do not exist non-empty subgraphs $G_1,G_2$ such that $G=G_1\vee G_2$.~Gross, Mansour, and Tucker \cite{2021_II} proved that partial Petrial polynomials respect the join operation.~Specifically,~$^{\partial}{\varepsilon^{\times}_{G_1\vee G_2}}(z)=^{\partial}{\varepsilon^{\times}_{G_1}}(z)\cdot ^{\partial}{\varepsilon^{\times}_{G_2}}(z)$.~Thus the study of partial Petrial polynomials for ribbon graphs reduces to its study for prime ribbon graphs.\\%~%Thus to study the partial Petrial polynomial of ribbon graph, it suffices to study that of prime ribbon graph. 
%Furthermore, they derived explicit formulas and recursive relations for certain families of ribbon graphs, including ladder graphs.
\indent A~\textit{bouquet}~is a ribbon graph with one vertex.~Yan and Li \cite{Petrial_Yan} showed that the partial Petrial polynomial of a bouquet is determined by its intersection graph.~This means that two bouquets with the same intersection graph have identical partial Petrial polynomials. \\%They also computed this polynomial for bouquets whose intersection graphs are either complete graphs or  paths. %Their results showed that the intersection graph is a complete graph $K_n$ if and only if the polynomial has non-zero coefficients for all terms from degree 1 to $n$.
%THEOREM:
%We include the known partial Petrial polynomial for complete graphs, which will be useful for comparison.
%Based on these findings, they posed the following problem: 
%If the partial Petrial polynomial of a connected graph $G$ is a binomial, must $G$ necessarily be a path? If not, can we characterize the graphs whose partial Petrial polynomials are binomials? 
%Recently, Feng, Yan and Zheng \cite{binomial_Yan} provided a positive answer for this problem.
%Recently, we extended the work of \cite{Petrial_Yan} by determining the partial Petrial polynomial for bouquets whose intersection graphs are cycles. We also provided a complete characterization of the prime bouquets for which the lowest degree among the nonzero coefficients in the partial Petrial polynomial is $2$. Furthermore, we explicitly computed $^{\partial}{\varepsilon_{B_n}^{\times}}(z)$ if the lowest-degree term of  $^{\partial}{\varepsilon_{B_n}^{\times}}(z)$ is of degree $2$. 
\indent A \emph{chord diagram} of degree $n$ consists of an oriented circle with $n$ chords whose $2n$ endpoints are all distinct. A \emph{framed} chord diagram assigns a framing in $\mathbb{Z}_2$ to each chord; chords with framing $0$ (resp. $1$) are called untwisted (resp. twisted). In this paper, solid (resp. dashed) lines represent framing $0$ (resp. $1$). %\indent A \emph{chord diagram} of degree $n$ is an oriented circle, together with $n$ chords of the circles, such that all of the $2n$ endpoints of the chords are distinct. A chord diagram is \emph{framed} if a map (a framing) from the set of chords to $\mathbb{Z}_2$ is given. The chords with framing $0$ (resp. $1$) are said to be untwisted (resp. twisted). In this paper, a solid line (resp. dashed line) represents a chord with framing $0$ (resp. $1$). 
Chord diagram is a concept in knot theory, which is an equivalent representation of a bouquet. In other words, a chord diagram uniquely corresponds to a bouquet. The functions on chord diagrams satisfying the four-term relation play a key role in the theory of Vassiliev knot invariants. In \cite{problemfourtermrelation}, Lando posed the following problem: 

\begin{Problem} \emph{\cite{problemfourtermrelation}}\label{Problem_4-invariants}
Which of the known graph invariants are $4$-invariants\emph{?}
\end{Problem}
\indent In 2021, Gross, Mansour, and Tucker \cite{2020_I} introduced the \textit{partial dual Euler genus polynomial}, which enumerates all partial duals of a ribbon graph according to their Euler genera. For this polynomial, \noindent Chmutov \cite{fourterm_Chmutov} and Deng et al. \cite{fourterm_DengYan}  proved that the four-term relation holds for orientable and non-orientable cases, respectively. Later, Cheng \cite{fourterm_chengzhiyun}  generalized the above  polynomial from ribbon graphs to arbitrary simple graphs  without referring to embeddings, and it was demonstrated that the generalized polynomial is $4$-invariant.\\
\indent Beck \cite{Beck} proved that the Euler genus of an orientable bouquet equals the rank of a certain matrix over $\mathbb{GF}(2)$. Later, Mellor \cite{weight} extended this result to all non-orientable bouquets. \\
\indent In this paper, we first give an auxiliary bouquet (Definition  \ref{Constructbouquet}) of any ribbon graph $G$ and prove that the Euler genus of $G$ equals the rank of the adjacency matrix of the signed intersection graph of this auxiliary bouquet over $\mathbb{GF}(2)$ (Theorem \ref{G=aux(G,T)}), which extends the result introduced by Beck and Mellor from bouquets to all ribbon graphs. Specifically, 

\begin{thm}\label{G=aux(G,T)}
	Let $G$ be a connected ribbon graph % with $n$ $(n\geq 2)$ vertices. 
	and let $T$ be any given spanning tree of $G$.~Then
	$$\varepsilon(G) = \text{rank}(adj(SI(Aux(G,T)))).%_{\mathbb{Z}_2}.
	$$
\end{thm}

Second, we derive an equivalent representation of the partial Petrial polynomial for all ribbon graphs. Specifically, the partial Petrial polynomial of a ribbon graph $G$ with $n$ vertices is equal to the sum of this polynomial for $2^{n-1}$ distinct bouquets, and further provide an explicit expression based on the rank of the matrix using the result introduced by Beck and Mellor (Theorem \ref{rankofpolyPetrial_ribbon}).

\begin{thm}\label{rankofpolyPetrial_ribbon}
	Let $G$ be a connected ribbon graph with $n$ $(n\geq 2)$ vertices. Assume that $T$ is  any spanning tree of $G$. Then the partial Petrial polynomial of $G$ has an equivalent expression %can be equivalently expressed as follows
	\[^{\partial}{\varepsilon_{G}^{\times}}(z)=\sum_{X \subseteq E(T)}  {}^{\partial}{\varepsilon_{Aux(G_X,T)}^{\times}}(z) =\sum_{X \subseteq E(T)} \sum_{Y \subseteq E(G)\backslash E(T)} z^{\text{{rank}}(adj(I(Aux(G_X,T)))+\mathbb{D}_Y)%_{\mathbb{Z}_2}
	},\]
	where $G_X$ is obtained from $G$ by applying a twist to all edges in $X$ and $\mathbb{D}_{Y}$ is defined as the $|E(G)| \times |E(G)|$ diagonal matrix whose $(i,i)$-th entry is $1$ if $i \in Y$ and $0$ otherwise.
\end{thm}

Furthermore, for the chord diagram corresponding to a bouquet, we give an equivalence relation for the partial Petrial polynomial (Theorem \ref{pPpva}), which is similar to the four-term relation. 

\begin{thm}\label{pPpva}
	Let $B$ be a chord diagram and let $a$ and $b$ be two adjacent chords of $B$. Then
	\[
	^{\partial}{\varepsilon_{B}^{\times}}(z) -
	^{\partial}{\varepsilon_{\widetilde{B}_{a,b}}^{\times}}(z) - ^{\partial}{\varepsilon_{\widetilde{B^{\prime}}_{a,b}}^{\times}}(z) +^{\partial}{\varepsilon_{B^{\prime}_{a,b}}^{\times}}(z)=0.
	\]
\end{thm}

Inspired by Theorem \ref{pPpva}, we define a modified partial Petrial polynomial by 
assigning coefficients +1 or -1 to the terms in the partial Petrial polynomial such that the modified polynomial satisfies the four-term relation (Definition \ref{modifiedpartialPetrial polynomial}). 
Later, we extend the concept of the modified partial Petrial polynomial from signed intersection graphs %\emph{circle graphs} 
to all simple signed graphs without referring to embeddings (Definition \ref{extendpPp}). %We prove that this extended  polynomial satisfies the four-term relation, . %This modified polynomial is determined by its intersection graph without referring to chord diagrams. 
Finally, we %develop the definition of the modified polynomial from circle graphs to all graphs and 
prove that the generalized modified partial Petrial polynomial is $4$-invariant (Theorem \ref{rankoffourterm}), which provides an answer to Problem \ref{Problem_4-invariants}.

\begin{thm}\label{rankoffourterm}
	The {modified partial Petrial polynomial} of any simple signed graph $S$ is $4$-invariant.
	%For any simple signed graph $S$ and any pair of vertices $a, b \in V(S)$, the modified partial Petrial polynomial $^{M\partial}{\varepsilon^{\times}_{S}}(z)$ of  $S$ satisfies the following four-term relation:
	%\[^{M\partial}{\varepsilon_{S}^{\times}}(z) -^{M\partial}{\varepsilon_{\widetilde{S}_{a,b}}^{\times}}(z) + ^{M\partial}{\varepsilon_{\widetilde{S^{\prime}}_{a,b}}^{\times}}(z) - ^{M\partial}{\varepsilon_{S^{\prime}_{a,b}}^{\times}}(z)=0.\]
\end{thm}

The rest of this paper is organized as follows. In Section $2$, we introduce the necessary notations and foundational properties used throughout the paper. In Section $3$, the proofs of Theorems \ref{G=aux(G,T)} and  \ref{rankofpolyPetrial_ribbon} are given. The proofs of Theorems \ref{pPpva} and \ref{rankoffourterm} are provided in Section~4.%In section $4$, we devote to the proofs of Theorems \ref{pPpva} and \ref{rankoffourterm}.

\section{Preliminary}
This section collects the foundational definitions and lemmas needed for the main results.  
Unless stated otherwise, all graphs are assumed to be finite and undirected. Throughout this paper, let $\mathbb{Z}_n$ be the set $\{1,2,\cdots,n\}$, and let all matrices be defined over $\mathbb{GF}(2)$. 
\subsection{Definitions and notations}
%In this subsection, we review several key definitions and notations that are essential for our development. \\
%\indent 
For a ribbon graph $G$, we denote its set of faces and the number of connected components by  $F(G)$ and $c(G)$, respectively. A ribbon graph is %considered 
{\em empty} if its edge set is empty. The {\em Euler characteristic} of $G$ is given by $\chi(G)=|V(G)|-|E(G)|+|F(G)|$, and its {\em Euler genus} is $\varepsilon(G)=2c(G)-\chi(G)$. %The orientable genus $\gamma(G)$ equals $\varepsilon(G)/2$ if $G$ is orientable and $\varepsilon(G)$ otherwise.  %\textcolor{red}{A ribbon graph with exactly one face is called a {\em quasi-tree}, while one with exactly two faces is a {\em nearly-quasi-tree}.} 
%A ribbon graph with exactly one face is called a \emph{quasi-tree} (a standard notion in the literature), while we introduce the term \emph{nearly-quasi-tree} for a ribbon graph with exactly two faces.
%A {\em bouquet} is a ribbon graph with a single vertex. 
%For a bouquet, % $B_n$, 
A loop of a bouquet is {\em twisted} (resp.{\em untwisted}) if it is homeomorphic to a M\"obius band (resp. an annulus).~The \emph{signed rotation system} of a ribbon graph consists of a cyclic ordering of the half-edges at each vertex-disc, together with a sign + or - assigned to each half-edge. In this system, the two half-edges of an untwisted loop carry the same sign (which may be omitted), whereas those of a twisted loop carry opposite signs.~%The {\em signed rotation system} of a ribbon graph consists of a cyclic ordering of the half-edges at each vertex-disc, with each half-edge assigned a sign, where the two half-edges of an untwisted loop have the same sign (the sign can be omitted), while those of a twisted loop have different signs. 
Two edges in a bouquet are {\em interlaced} if their half-edges alternate in the signed rotation.%~A loop is {\em trivial} if it is not interlaced with any other loops.
%The orientable and nonorientable bouquets with Euler genus $2$ are called {\em toroidal bouquets} and {\em Klein bottle bouquets}, respectively.
%\textcolor{blue}{The orientable bouquets with Euler genus $0$ and $2$ is called {\em plane bouquets} and {\em toroidal bouquets}, respectively. The nonorientable bouquets with Euler genus %$1$ and $2$ are called %{\em projective planar bouquets} and {\em Klein bottle bouquets}, respectively.}

A \emph{signed graph} is a graph whose vertices are labeled with the sign $+$ or $-$. A \emph{simple signed graph} is a signed graph whose underlying graph (ignoring signs of vertices) is simple. The \textit{intersection graph} $I(B)$ of a bouquet $B$ has the loops of $B$ as its vertex set, and an edge exists between two vertices if the corresponding loops are interlaced. The \emph{signed intersection graph} $SI(B)$ is a variant of $I(B)$ where each vertex is assigned a sign corresponding to its loop type $+$ for untwisted, $-$ for twisted. %A graph is a \emph{circle graph} if it is the intersection graph of some bouquets. 
%The partial Petrial polynomial of a circle graph $G$, denoted by $P_G^{\times}(z)$, is defined as $P_G^{\times}(z):={\partial}{\varepsilon^{\times}_{B}}(z)$, where $B$ is a bouquert such that $G=I(B)$. 
The \emph{adjacency matrix} $\text{adj}(SI(B))$ of the signed intersection graph of the bouquet $B$ is obtained from the adjacency matrix $\text{adj}(I(B))$ of the intersection graph $I(B)$ by assigning $1$ to the diagonal entries corresponding to vertices of $SI(B)$ with sign $-$.\\
\indent A function $f$ is said to satisfy the \emph{four-term relation} (\cite{circle_graph}, \cite{deffourterm1}, \cite{def2fourterm1}) if the alternating sum of the values of $f$ is zero on the following quadruples of diagrams as shown in Fig. \ref{Fig_four-term}.\\
\indent Let us associate each pair of (distinct) vertices $a,b\in V(G)$ of a simple graph $G$ with two other graphs $G^{\prime}_{a,b}$ and $\widetilde{G}_{a,b}$  \cite{problemfourtermrelation}. The graph $G^{\prime}_{a,b}$ is obtained from $G$ by erasing the edge $ab\in E(G)$ in the
case this edge exists, and by adding the edge otherwise. The graph $\widetilde{G}_{a,b}$ is obtained from $G$ in the following way. For any vertex $c\in V(G)\backslash\{a,b\}$, we change the adjacency between $a$ and $c$ if $c$ is adjacent to $b$, and all other edges do not change.

\begin{figure}[htbp]
	\centering
	\includegraphics[%height=6.6cm,
	width=0.6\textwidth]{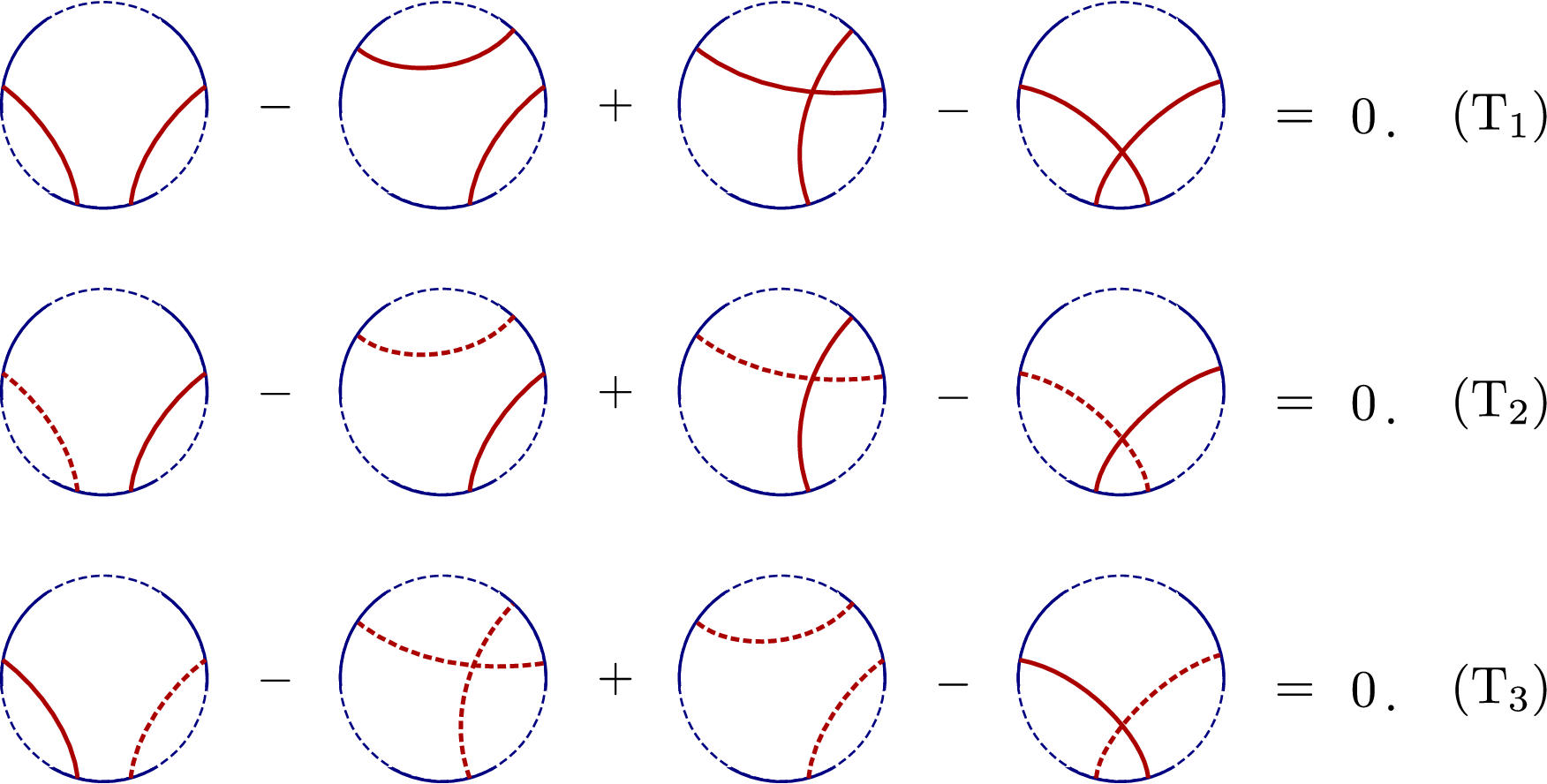}
	\caption{Four-term relation for framed chord diagrams.}	\label{Fig_four-term}
\end{figure}

%\begin{defi}\emph{\cite{problemfourtermrelation}}\label{4-invariant}
A graph invariant $g$ is called $4$-\emph{invariant} \cite{problemfourtermrelation} if it satisfies the four-term relation
\[g(S)-g(\widetilde{S}_{a,b})+g(\widetilde{S}^{\prime}_{a,b})-g(S^{\prime}_{a,b})=0\]
for each graph $S$ and for any pair $a,b$ of its vertices.

%For a polynomial $g(x)=\sum_{i=0}^{n}a_ix^{i}$, the {\em support} of the polynomial $g(x)$ is the set $supp(g)=\{i|a_i\neq 0\}$ of indices of the non-zero coefficients. Let $s=\min {supp(g)}$, $t=\max {supp(g)}$. % and $H=\{j|s\leq j\leq t,j\in \mathbb{Z}\}$. The polynomial $g(x)$ is called interpolating if $supp(g)=\{j|s\leq j\leq t,j\in \mathbb{Z}\}$.

\subsection{Key lemmas}

In this subsection, we provide some key lemmas for our arguments.

We now recall the well-known formula for calculating the Euler genus of a ribbon graph.

\begin{lem}\emph{(\cite{Yanqi_Intersection})}\label{varepsilon_ribbon}
Let $G$ be a ribbon graph. We have $\varepsilon(G)=2c(G)-\chi(G).$%=2c(G)-|V(G)|+|E(G)|-|F(G)|$.
\end{lem}

For the special case of a bouquet, this formula simplifies considerably.

\begin{lem}\emph{(\cite{Yanqi_Intersection})}\label{varepsilon}
If $B$ is a bouquet, then $\varepsilon(B)=1+|E(B )|-|F(B)|$.
\end{lem}

Lemma \ref{slide} shows that the sliding operation does not change the Euler genus of the surface.

\begin{lem}\emph{\cite{fourterm_Chmutov}}\label{slide}
Sliding one edge-ribbon of a ribbon graph along another one results in a homeomorphic surface. Hence the two ribbon graphs have the same Euler genus.
\end{lem}

The intersection graph plays a crucial role in the classification of bouquets.

\begin{lem}\label{varepsilion_SIB}\emph{\cite{Yanqi_Intersection}}
	For two bouquets $B_1$ and $B_2$, if $SI(B_1)=SI(B_2)$, % have the same signed intersection graph, 
	then $\varepsilon(B_1)=\varepsilon(B_2)$.
\end{lem}

Lemma \ref{Petrial_orientable} demonstrates a crucial connection between the partial Petrial polynomial and the intersection graph of a bouquet.

\begin{lem}\emph{\cite{Matroids_Yan}}\label{Petrial_orientable}
If bouquets $B_1,B_2$ have the same intersection graph, then $^{\partial}{\varepsilon^{\times}_{B_1}}(z)=^{\partial}{\varepsilon^{\times}_{B_2}}(z)$.
\end{lem}

Lemma \ref{G,GA} shows that studying the partial Petrial polynomial of a ribbon graph is equivalent to studying the polynomial of any of its partial Petrial graphs.

\begin{lem}\emph{\cite{Petrial_Yan}}\label{G,GA}
Let $G$ be a ribbon graph and $A\subseteq E(G)$. Then  $^{\partial}{\varepsilon^{\times}_{G}}(z)=^{\partial}{\varepsilon^{\times}_{G^{\times|A}}}(z)$.
\end{lem}

%Lemma \ref{additionofjoin} shows that how the Euler genus behaves with respect to the join operation.
%\begin{lem}\emph{\cite{2001_Graphs_on_surfaces}}\label{additionofjoin}
%Let $G_1$ and $G_2$ be disjoint ribbon graphs, then
%$\varepsilon(G_1\vee G_2)=\varepsilon(G_1)+\varepsilon(G_2)$.
%\end{lem}

%\vspace{-0.2em}
%The concept of a prime bouquet is characterized by the connectivity of its intersection graph.

%\vspace{-0.2em}
%\begin{lem}\emph{(\cite{Yanqi_Intersection})}\label{prime-conected}
%A bouquet $B_n$ is prime if and only if its  intersection graph is connected.	
%\end{lem}

%A key property of the partial Petrial polynomial is that it is interpolating, meaning its coefficients are non-zero for all integers between the minimum and maximum degrees.

%\begin{lem}\emph{(\cite{2021_II})}\label{interpolation of Petrial} 
%For any ribbon graph $G$, the partial Petrial polynomial $^{\partial}{\varepsilon^{\times}_{G}}(z)$ is interpolating.
%\end{lem}

%We also rely on the result for the highest degree of the partial Petrial polynomial.% for a connected graph.

%\begin{lem}\emph{(\cite{2021_II})}\label{highest-term of Petrial} 
%If $G$ is connected, then the highest degree of $^{\partial}{\varepsilon^{\times}_{G}}(z)$ is $|E(G)|-|V(G)| + 1$.
%\end{lem}

In this paper, the rank of any matrix $A$ over $\mathbb{GF}(2)$ is denoted by $\text{rank}(A)$.
%The rank of a matrix $A$ over $\mathbb{GF}(2)$ is denoted $\text{rank}(A)_{\mathbb{Z}_{2}}$. %Finally, we use a known result that relates the rank of this matrix to the number of edges and faces of the bouquet.

\begin{lem}\emph{\cite{weight}}\label{rank=1+e-f}
For a bouquet $B$, then $\text{rank}(\text{adj}(SI(B)))%_{\mathbb{Z}_2}
=|E(B)|-|F(B)|+1$.
\end{lem}

By Lemmas \ref{varepsilon} and \ref{rank=1+e-f}, we have Lemma \ref{rank=genusofbouquet}, which is a key lemma for our main results.

\begin{lem}\label{rank=genusofbouquet}
For a bouquet $B$, then $\varepsilon(B)=\text{rank}(\text{adj}(SI(B)))%_{\mathbb{Z}_2}
$.
\end{lem}

Based on Lemma \ref{rank=genusofbouquet}, we gave an equivalent definition of the partial Petrial polynomial from the perspective of the matrix rank as follows.
\begin{lem}\label{Petrialofbouquet-rank}
	%For a $B$, then t
	The partial Petrial polynomial of any prime bouquet $B$ has an exact expression:  %can be expressed equivalently as follows
	\[^{\partial}{\varepsilon_{B}^{\times}}(z)=\sum_{A \subseteq E(B)} z^{\text{rank}(\text{adj}(SI(B^{\times|A})))%_{\mathbb{Z}_2}
	} =\sum_{A \subseteq \mathbb{Z}_{n}} z^{\text{rank}(\text{adj}(I(B))+\mathbb{D}_{A})}.%_{\mathbb{Z}_2}}.
	\]
\end{lem}

It is well known that ribbon graphs can be viewed as surfaces with boundary. We review an equivalent condition for surface homeomorphism.

\begin{lem}\emph{\cite{Topology of Surfaces}} \label{thm:compact-surface-equivalence}
	Let \( S_1 \) and \( S_2 \) be compact connected surfaces with boundary. Then \( S_1 \) is topologically equivalent to \( S_2 \) if and only if they have the same number of boundary components, both are orientable or both non-orientable, and they have the same Euler characteristic.
\end{lem}

%\vspace{-0.8em}
\section{An equivalent definition of partial Petrial polynomial}

In \cite{Beck} and \cite{weight}, Beck and Mellor proved, in the orientable and non-orientable cases respectively, that the Euler genus of a bouquet equals the rank of a certain matrix over $\mathbb{GF}(2)$.  In this section, we extend this result to all ribbon graphs with more than one vertex. Furthermore, we give an explicit expression of the partial Petrial polynomial. 

\subsection{A key construction and a proof of Theorem \ref{G=aux(G,T)}}

\begin{defi}\emph{\cite{2001_Graphs_on_surfaces}}
For any connected ribbon graph $G$ with at least two vertices, let $e$ be a proper edge-ribbon of $G$ connecting two distinct vertex-discs $u$ and $v$. The edge-contraction $G/e$ is defined by replacing $e,u$ and $v$ with a new vertex-disc $e\cup u\cup v$ and keeping other edge-ribbons and vertex-discs unchanged $($see Fig.  \emph{\ref{step2}} for $G/e$, where  $B^{-1}$ is obtained from the ordered set $B$ by reversing $B$ and then multiplying all elements in $B$ by $-1$$)$.
\end{defi}

\begin{defi}\label{Constructbouquet}
	For any connected ribbon graph $G$ with $n$ $(n\geq 2)$ vertices, let $T$ be any spanning tree of $G$. The auxiliary bouquet, denoted by $Aux(G,T)$, of $G$ with respect to $T$ is obtained from $G$ by contracting all edge-ribbons of $T$.
\end{defi}
\vspace{-1.2em}
\begin{Remark}\label{auxtwist}
	For a ribbon graph $G$, assume that $T$ is any spanning tree of $G$ and $X\subseteq E(T), Y\subseteq E(G)\backslash E(T)$. Let $G_X$ be the ribbon graph obtained from $G$ by applying a twist to all edge-ribbons in $X$. Then $Aux((G_X)^{\times|Y},T)=Aux(G_X,T)^{\times|Y}$ by Definition \emph{\ref{Constructbouquet}}.
\end{Remark}
\vspace{-0.3em}
\begin{figure}[htbp]
	\centering
	\includegraphics[width=0.98\textwidth]{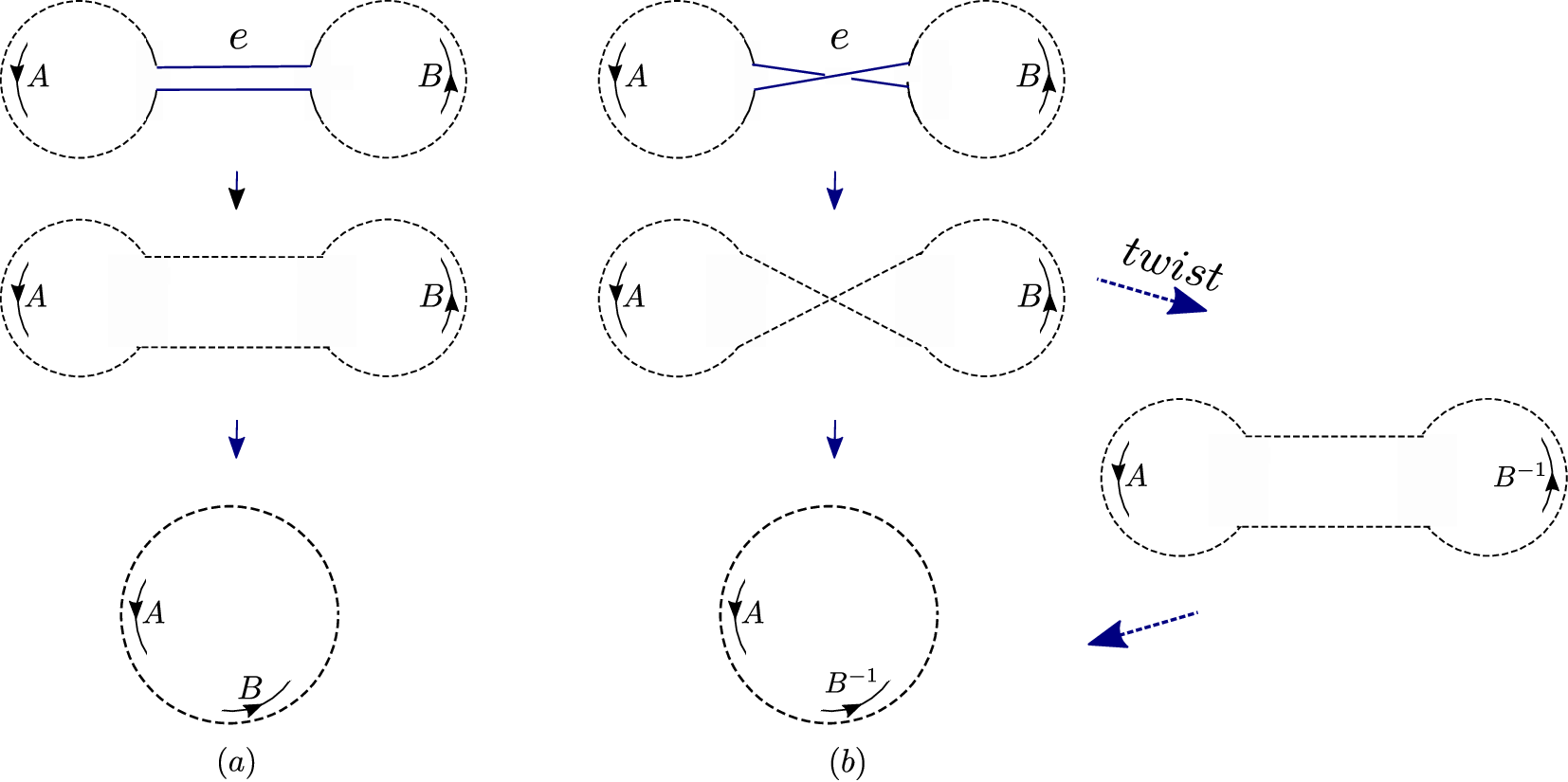}
	\caption{The ribbon graph $G/e$, where the edge $e$ in ($a$) (resp. ($b$)) is untwisted (resp. twisted).}\label{step2}
\end{figure}
\vspace{-1em}
\begin{lem}\label{ver(G)=ver(aux(G,T))}
	Let $G$ be a connected ribbon graph. Assume that $T$ is a spanning tree of $G$. Then
	$$\varepsilon(G) = \varepsilon(Aux(G,T)).$$
\end{lem}
\begin{proof}
Recall that the ribbon graph $Aux(G,T)$ is obtained from $G$ by contracting all edge-ribbons in $T$. Since ribbon graphs are surfaces with boundaries, the surface $Aux(G,T)$ can be seen as being obtained from the surface $G$ by continuously deforming all edge-ribbons in $T$ such that $E(T)\cup V(G)$ becomes the unique vertex-disc of $Aux(G,T)$ (see Fig.\ref{step2}). Thus the ribbon graphs $G$ and $Aux(G,T)$ are homeomorphic through continuous deformation from the perspective of surfaces. According to Lemma \ref{thm:compact-surface-equivalence}, two homeomorphic surfaces have the same Euler characteristic, which implies that $\chi(G)=\chi(Aux(G,T))$. It follows that $\varepsilon(G)=\varepsilon(Aux(G,T))$ by Lemma \ref{varepsilon_ribbon}. 
\end{proof}

\noindent\emph{\textbf{Proof of Theorem \emph{\ref{G=aux(G,T)}}}~~} Since $Aux(G,T)$ is a bouquet, one has $$\varepsilon(Aux(G,T))=\text{rank}(adj(SI(Aux(G,T))))%_{\mathbb{Z}_2}
$$ by Lemma \ref{rank=genusofbouquet}.
 Combined with Lemma \ref{ver(G)=ver(aux(G,T))}, Theorem \ref{G=aux(G,T)} holds.
\hfill{$\square$}

\subsection{A proof of Theorem \ref{rankofpolyPetrial_ribbon}}

\vspace{0.6em}
\noindent\emph{\textbf{Proof of Theorem \emph{\ref{rankofpolyPetrial_ribbon}}}~~}
For any ribbon graph $G$, there is a $Z\subseteq E(G)$ such that the partial Petrial $G^{\times|Z}$ with all edges untwisted. Since  $^{\partial}{\varepsilon^{\times}_{G}}(z)=^{\partial}{\varepsilon^{\times}_{G^{\times|Z}}}(z)$ as Lemma \ref{G,GA}, we assume that all edges of $G$ are untwisted. 
For each $X\subseteq E(T)$, define $V_X=\{X\cup Y \mid Y\subseteq T^{c}\},$ 
where $T^{c}=E(G)\backslash E(T)$. Then the collection $\{ V_{X} \mid X \subseteq E(T) \}$ forms a partition of the power set of $E(G)$ into $2^{n-1}$ disjoint sets. By Definition \ref{poly_Petrial}, one has $$^{\partial}{\varepsilon^{\times}_{G}}(z)=\sum\limits_{A\subseteq E(G)}z^{\varepsilon{(G^{\times|A})}}=\sum_{X \subseteq E(T)} \sum_{A \in V_X}z^{\varepsilon{(G^{\times|A})}}
=\sum_{X \subseteq E(T)} \sum_{Y \subseteq T^c}
z^{\varepsilon{(G^{\times|(X\cup Y)})}}.$$ 
Recall that $G_X$ is obtained from $G$ by applying a twist to all edge-ribbons in $X$, we have that $G^{\times|(X\cup Y)}=(G_X)^{\times|Y}$, which implies that $\varepsilon(G^{\times|(X\cup Y)})=\varepsilon((G_X)^{\times|Y}).$ Since the auxiliary graph of $(G_X)^{\times|Y}$ with respect to $T$ is $Aux((G_X)^{\times|Y},T)$, we have that $\varepsilon((G_X)^{\times|Y})=\varepsilon(Aux((G_X)^{\times|Y},T))$ 
by Lemma \ref{ver(G)=ver(aux(G,T))}. According to Remark \ref{auxtwist}, one has $Aux((G_X)^{\times|Y},T)=Aux(G_X,T)^{\times|Y}$, which implies that  $\varepsilon(G^{\times|(X\cup Y)})=\varepsilon((Aux(G_X,T))^{\times|Y})$. Thus $$^{\partial}{\varepsilon^{\times}_{G}}(z)=\sum_{X \subseteq E(T)} \sum_{Y \subseteq T^c}
z^{\varepsilon{(G^{\times|(X\cup Y)})}}=
\sum_{X \subseteq E(T)} \sum_{Y \subseteq T^c}
z^{\varepsilon((Aux(G_X,T))^{\times|Y})}=
\sum_{X \subseteq E(T)}~^{\partial}{\varepsilon^{\times}_{Aux(G_X,T)}}(z).$$ 
%We divide the family of all the subsets of $E(G)$ into the disjoint union of $2^{n-1}$ parts based on any $T^{\prime}\subseteq T$ such that $V_{T^{\prime}}=\{A\subseteq E(G)~|~T^{\prime}\subseteq A\}$.

 Since $aux(G_X,T)$ is a bouquet, by Lemma \ref{Petrialofbouquet-rank}, we have that
\[\sum_{X \subseteq E(T)}~^{\partial}{\varepsilon^{\times}_{Aux(G_X,T)}}(z)=\sum_{X \subseteq E(T)}\sum_{Y \subseteq T^c} z^{\text{rank}(\text{adj}(I(Aux(G_X,T)))+\mathbb{D}_{Y})%_{\mathbb{Z}_2}
}.\]
The proof is finished.
\hfill{$\square$}

Next, we introduce an algorithm to calculate the partial Petrial polynomial by Theorem \ref{rankofpolyPetrial_ribbon}.

\begin{algorithm}[H]\label{alg1}
	\small
	\caption{~}\label{algorithm2}
	\begin{algorithmic}[The partial Petrial polynomial of  a ribbon graph]
		\REQUIRE Given an orientable connected ribbon graph $G$ with $m$ edges and $n$ vertices ($n\geq 2$) and an edge subset $T$ of a spanning tree of $G$.\\
		\hspace{-1.7em}\textbf{Input}  The rotation system of $G$; the edge subset $T$. \\%; %, where $T$ is the set of all twisted edges of ${B_n}$; an integer $k=0$.  \\
		\hspace{-1.7em}\textbf{Output} The partial Petrial polynomial  $^{\partial}{\varepsilon^{\times}_{G}}(z)$ of $G$.  \\
		\hspace{-1.7em}\noindent \textbf{01:} Set integers  $k=1$, $l=1$,  $^{\partial}{\varepsilon^{\times}_{G}}(z)=0$ and sign $z$.\\%Let $U=\emptyset$, $P$ be the set of all untwisted edges of $B_n$ interlaced with $T$, $X=(T\cup P\cup U)^{c}$.\\
		%\hspace{-1.7em}\noindent \textbf{2: while $X\neq\emptyset$ do}.\\
		\hspace{-1.7em}\noindent \textbf{02: }Get all subsets of $T$ (resp. $\mathbb{Z}_{m-n+1}$) and store them in the matrix $\mathbb{T}$ (resp. $\mathbb{U}$).\\
	    \hspace{-1.7em}\noindent \textbf{03: For}  $1\leq k\leq 2^{n-1}$ \textbf{do}\\
	    \hspace{-1.7em}\noindent \textbf{04:}\hspace{1.1em} %\textbf{Update} 
	    Generate the signed rotation of the bouquet $aux(G,\mathbb{T}_k)$ and store it in matrix $X_k$.\\
	    \hspace{-1.7em}\noindent \textbf{05:}\hspace{1.1em}
	    Generate the adjacency matrix $M_k$ of the intersection graph of $X_k$.\\
	    \hspace{-1.7em}\noindent \textbf{06:}\hspace{1.1em}
	     \textbf{For}  $1\leq l\leq 2^{m-n+1}$ \textbf{do}\\
	    \hspace{-1.7em}\noindent \textbf{07:}\hspace{2.2em}
	    Calculate $a_l=\text{rank}(M_k+\mathbb{D}_{\mathbb{U}_l})$ over $\mathbb{GF}(2)$.\\
	    \hspace{-1.7em}\noindent \textbf{08:}\hspace{2.2em}
	    Update $^{\partial}{\varepsilon^{\times}_{G}}(z)=^{\partial}{\varepsilon^{\times}_{G}}(z)+z^{a_l}$, $l=l+1$.\\
	    \hspace{-1.7em}\noindent   \textbf{09:}\hspace{1.1em} \textbf{end for}.\\
		\hspace{-1.7em}\noindent \textbf{10:}\hspace{1.1em} \textbf{Update} $k\leftarrow k+1$;\\
		\hspace{-1.7em}\noindent \textbf{11: end for}.\\
		\hspace{-1.7em}\textbf{12: return}~$^{\partial}{\varepsilon^{\times}_{G}}(z)$.\\
		%\indent{\textbf{end}}\\
		%\ENSURE~~
	\end{algorithmic}
\end{algorithm}

We can calculate $^{\partial}{\varepsilon^{\times}_{G}}(z)$ for some ribbon graphs as examples by using this algorithm with the help of computer. For example, let the rotation system of $G$ be $\{\rho_{v_i}\mid i\in \mathbb{Z}_6\}$, where 
$\rho_{v_1}:(1,8,12), \rho_{v_2}:(9,4,2,3,8), \rho_{v_3}:(11,10,5,6,9), \rho_{v_4}:(7,4,11), \rho_{v_5}:(5,2,1,6,12), \rho_{v_6}:(3,7,10)$, then one has $^{\partial}{\varepsilon^{\times}_{G}}(z)=1412z^7 + 1692z^6 + 779z^5 + 189z^4 + 23z^3 + z^2$.

\section{A four-term relation for chord diagram }
In this section, we first present an equivalence relation on four chord diagrams with respect to their partial Petrial polynomials, which is similar to the four-term relation. Then we define a modified partial Petrial polynomial by assigning  coefficients +1 or -1 to every term in the partial Petrial polynomial such that the resulting polynomial satisfies the four-term relation of graphs.

\subsection{A proof of Theorem \ref{pPpva}}
First, we introduce two types of transformations on chord diagrams.
\vspace{-0.3em}
\begin{defi}
	For any chord diagram $B$ and chords $a$ and $b$ of $B$, if one endpoint of $a$ and one endpoint of $b$ are adjacent in the signed rotation of $B$, we call  chords $a$ and $b$ are adjacent in $B$.
\end{defi}
\vspace{-0.8em}
\begin{figure}[htbp]
	\centering
	\includegraphics[%height=6.6cm,
	width=0.56\textwidth]{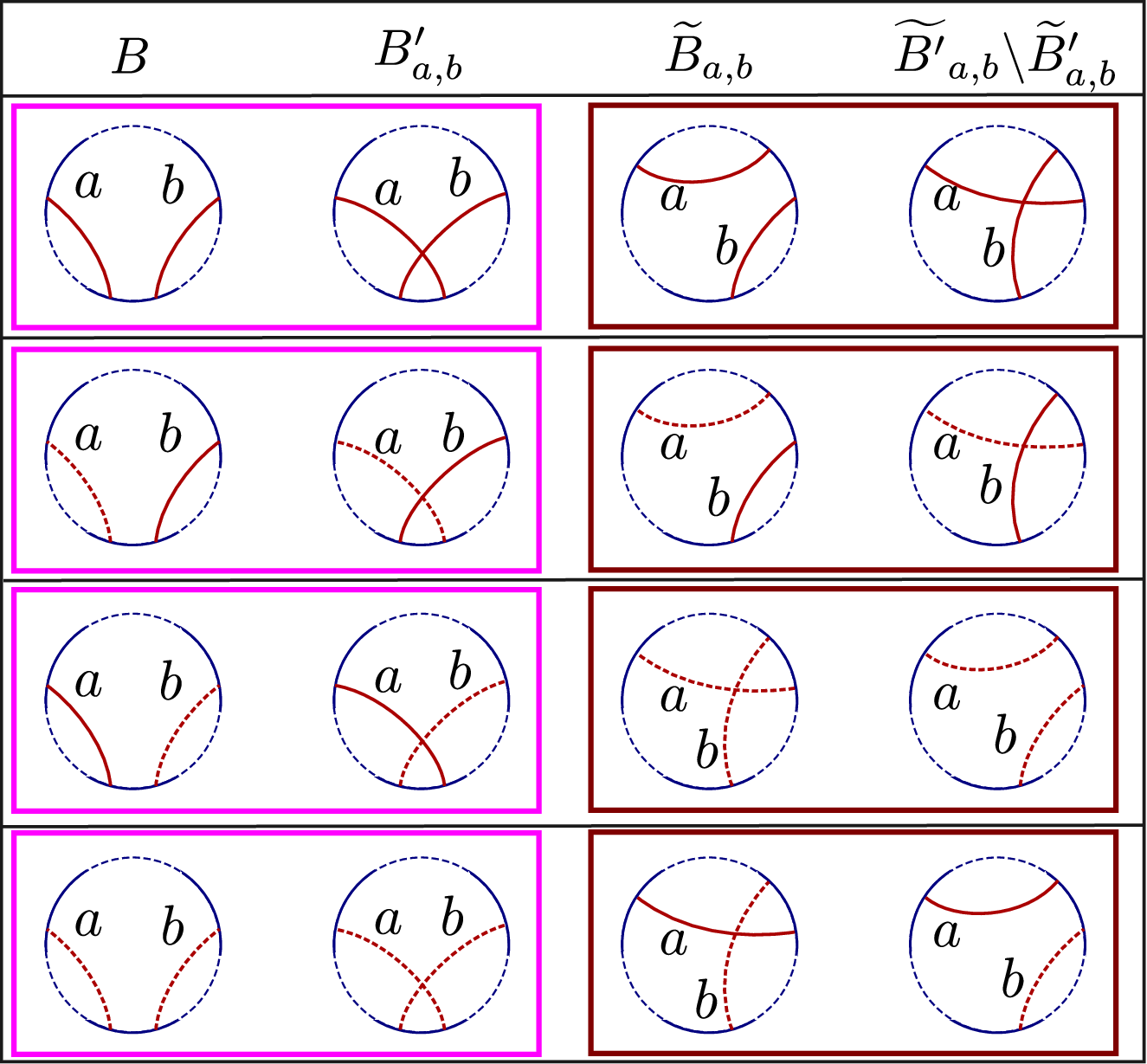}
	\caption{Three kinds of transformations of the chord diagram $B$.}	\label{B_slide_exchange}
\end{figure}

\vspace{-1.2em}
\begin{defi}
	For any chord diagram $B$, let $a$ and $b$ be two adjacent chords of $B$. We define two transformations of $B$ as follows (see Fig. \emph{\ref{B_slide_exchange}}):
	\noindent  The chord diagram ${B}^{\prime}_{a,b}$ is obtained from $B$ by exchanging the two adjacent endpoints of chords $a$ and $b$, and the chord diagram $\widetilde{B}_{a,b}$ is obtained from $B$ by sliding chord $a$ along $b$ at the endpoint of $a$ that is adjacent to $b$.
\end{defi}
\begin{Remark}
	It is easy to check that $(\widetilde{B}_{a,b})^{\prime}_{a,b}=\widetilde{(B^{\prime}_{a,b})}_{a,b}$. We use $\widetilde{B^{\prime}}_{a,b}$ or $\widetilde{B}^{\prime}_{a,b}$ to denote the composition of these two operations in this paper.
\end{Remark}

\vspace{-1em}
\begin{Remark}
	Since chord diagrams $\widetilde{B}_{a,b}$,  $B^{\prime}_{a,b}$ and $\widetilde{B}^{\prime}_{a,b}$ can be obtained from $B$, there exist one-to-one correspondences between these edge sets. Thus we do not distinguish them.
\end{Remark}

\noindent\emph{\textbf{Proof of Theorem \emph{\ref{pPpva}}}~~}
We will prove that 
\begin{equation}\label{eq3.1}
^{\partial}{\varepsilon_{B}^{\times}}(z) + ^{\partial}{\varepsilon_{B^{\prime}_{a,b}}^{\times}}(z) = 
^{\partial}{\varepsilon_{\widetilde{B}_{a,b}}^{\times}}(z)+ ^{\partial}{\varepsilon_{\widetilde{B^{\prime}}_{a,b}}^{\times}}(z).
\end{equation}
By the definition of $B^{\prime}_{a,b}$, we know that exactly one of $B$ and $B^{\prime}_{a,b}$ satisfies that edges $a$ and $b$ are interlaced. Since $(B^{\prime}_{a,b})^{\prime}_{a,b}=B$, to prove Equation (\ref{eq3.1}), it suffices to prove Equation (\ref{eq3.1}) under the assumption that $a$ and $b$ are not interlaced in $B$. Based on whether \(a\) and \(b\) are twisted or not, ignoring all edges in $B$ except for $a$ and $b$, the graph $B$ falls into the following four cases (corresponding to the first column of Fig. \ref{B_slide_exchange}). In each of the four cases of $B$, the graphs $B^{\prime}_{a,b}$, $\widetilde{B}_{a,b}$ and $\widetilde{B^{\prime}}_{a,b}$ are uniquely determined (corresponding to each row of  Fig. \ref{B_slide_exchange}). For four different cases of $B$, we know that $I(B)\cup I(B^{\prime}_{a,b})$ (resp. $I(\widetilde{B}_{a,b})\cup I(\widetilde{B^{\prime}}_{a,b})$) are always the same (see Fig. \ref{B_slide_exchange}, where the bouquets in boxes with the same color have the same intersection graphs). By Lemma \ref{Petrial_orientable}, it suffices to consider the case in which all edges of $B$ are untwisted (see first row of Fig. \ref{B_slide_exchange}). %It is easy to see that the bouquets $B^{\prime}_{a,b},\widetilde{B}_{a,b},\widetilde{B^{\prime}}_{a,b}$ are orientable. 
%Since  $B,B^{\prime}_{a,b},\widetilde{B}_{a,b},\widetilde{B^{\prime}}_{a,b}$ have the same edges, based on $a,b\in E(B)$, w

%\vspace{-0.3em}
\begin{figure}[htbp]
	\centering
	\includegraphics[%height=6.6cm,
	width=0.85\textwidth]{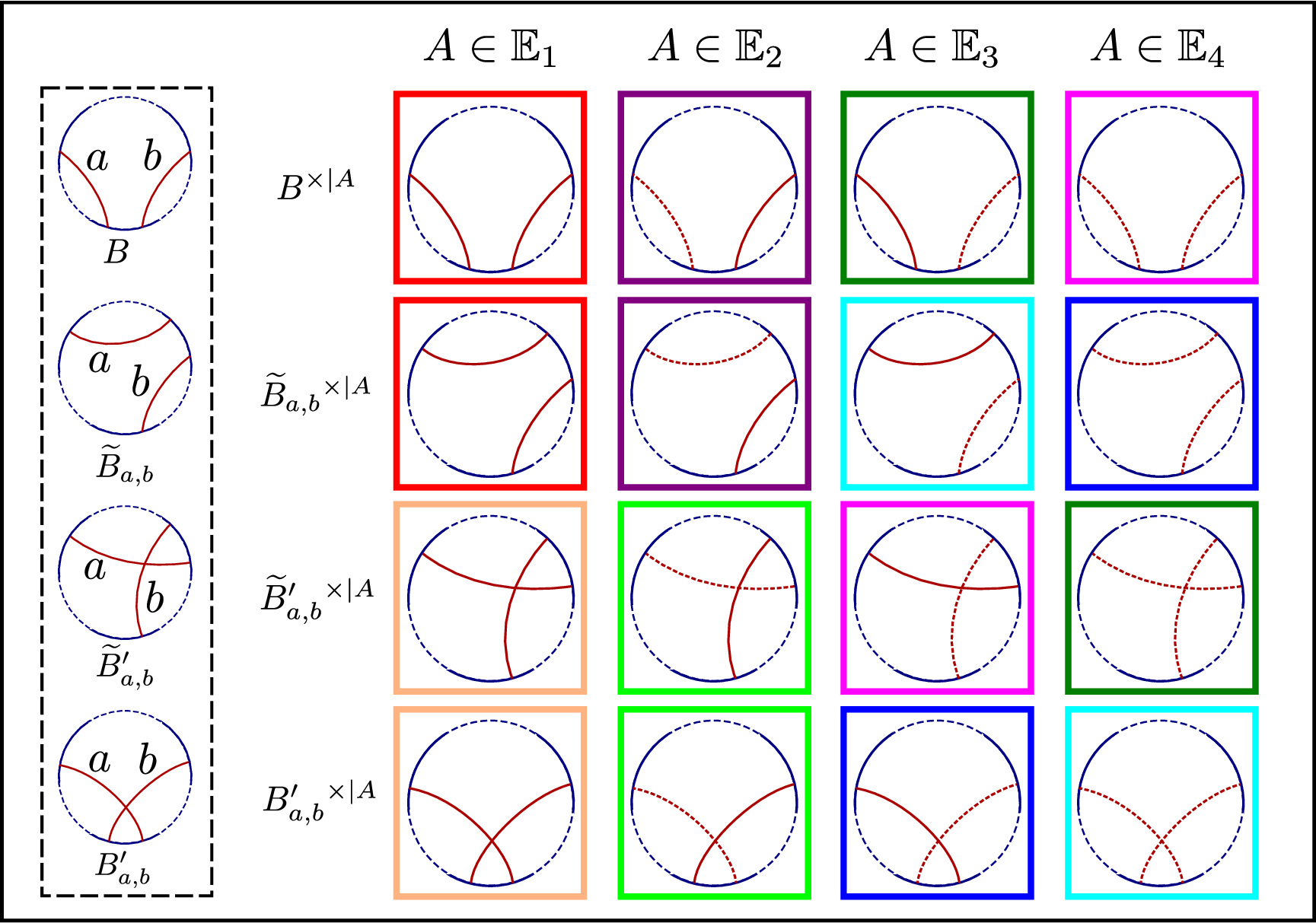}\\
	\caption{The partial Petrial graphs $B^{\times|A},\widetilde{B}_{a,b}{}^{\times|A},\widetilde{B}^{\prime}_{a,b}{}^{\times|A}$, and $B^{\prime}_{a,b}{}^{\times|A}$, where two bouquets in boxes of the same color can be obtained by sliding one edge.}%The proof of four-term relations of ($T_1$).}
\label{Fig_T_1}
\end{figure}

\vspace{-0.8em}
We divide the family of all the subsets of $E(B)$ into the disjoint union of four parts $\mathbb{E}_1\cup \mathbb{E}_2 \cup \mathbb{E}_3 \cup \mathbb{E}_4$ based on whether a subset contains $a$ or $b$, where
$\mathbb{E}_1 = \{A \subseteq E(B)\mid a\notin A, b \notin A\}$,
$\mathbb{E}_2 = \{A \subseteq E(B)\mid a\in A, b \notin A\}$,
$\mathbb{E}_3 = \{A \subseteq E(B)\mid a\notin A, b \in A\}$,
$\mathbb{E}_4 = \{A \subseteq E(B)\mid a\in A, b \in A\}$. Since edge-ribbon sliding operation keeps the genus of surface as Lemma \ref{slide}, we have  (see Fig.\ref{Fig_T_1})

\vspace{-1.5em}
\begin{align}\label{equationbyslide}
\begin{cases}
\varepsilon(B^{\times|A})=\varepsilon({\widetilde{B}_{a,b}}{}^{\times|A}),
\varepsilon({{B}^{\prime}_{a,b}}^{\times|A})=
\varepsilon({\widetilde{B}^{\prime}_{a,b}}{}^{\times|A}),
& \text{if~} A\in \mathbb{E}_1 \text{~or~} A\in \mathbb{E}_2;\\
%\varepsilon(B^{\times|A})=\varepsilon({\widetilde{B}^{\prime}_{a,b}}{}^{\times|(A\Delta \{a\})}),\varepsilon({B^{\prime}_{a,b}}{}^{\times|A})=\varepsilon({\widetilde{B}_{a,b}}^{\times|(A\Delta\{a\})}),& \text{if~} A\in E_3;\\
\varepsilon(B^{\times|A})=\varepsilon({\widetilde{B}^{\prime}_{a,b}}{}^{\times|(A\Delta \{a\})}),
\varepsilon({{B}^{\prime}_{a,b}}{}^{\times|A})=\varepsilon({\widetilde{B}_{a,b}}{}^{\times|(A\Delta\{a\})}),
& \text{if~} A\in \mathbb{E}_3 \text{~or~} A\in \mathbb{E}_4.
\end{cases}
\end{align}

\vspace{-0.3em}
Thus, one has that
\vspace{-0.3em}
\[\begin{aligned}
	^{\partial}{\varepsilon^{\times}_{B}}(z) + ^{\partial}{\varepsilon_{B^{\prime}_{a,b}}^{\times}}(z)&=
	\sum\limits_{A\subseteq E(B)}z^{\varepsilon(B^{\times|A})} + \sum\limits_{A\subseteq E(B^{\prime}_{a,b})}z^{\varepsilon({{B}^{\prime}_{a,b}}^{\times|A})}&  \text{Definition}~\ref{poly_Petrial}\\
	&=\sum\limits_{A\in \mathbb{E}_1\cup \mathbb{E}_2\cup \mathbb{E}_3\cup \mathbb{E}_4} (z^{\varepsilon(B^{\times|A})}+z^{\varepsilon({B^{\prime}_{a,b}}^{\times|A})}
	)\\
	&=\sum\limits_{A\in \mathbb{E}_1\cup \mathbb{E}_2} (z^{\varepsilon(\widetilde{B}_{a,b}{}^{\times|A})}+z^{\varepsilon({\widetilde{B}^{\prime}_{a,b}}{}^{\times|A})}
	)\\
	&+\sum\limits_{A\in \mathbb{E}_3\cup \mathbb{E}_4} (z^{\varepsilon(\widetilde{B}_{a,b}{}^{\times|(A\Delta\{a\})})}+z^{\varepsilon({\widetilde{B}^{\prime}_{a,b}}{}^{\times|(A\Delta\{a\})})})
	&\text{Equation } (\ref{equationbyslide})\\
	&=\sum\limits_{A\in \mathbb{E}_1\cup \mathbb{E}_2\cup \mathbb{E}_3\cup \mathbb{E}_4} (z^{\varepsilon(\widetilde{B}_{a,b}{}^{\times|A})}+z^{\varepsilon({\widetilde{B}^{\prime}_{a,b}}{}^{\times|A})}
	)
	& 3,4\text{-th column of Fig}.\ref{Fig_T_1}\\
	&=
	^{\partial}{\varepsilon^{\times}_{\widetilde{B}_{a,b}}}(z) + ^{\partial}{\varepsilon_{\widetilde{B}^{\prime}_{a,b}}^{\times}}(z).
	&  \text{Definition}~\ref{poly_Petrial}
\end{aligned} \]

\vspace{-0.3em}
It is finished.
%$$^{\partial}{\varepsilon^{\times}_{B}}(z) + ^{\partial}{\varepsilon_{B^{\prime}_{a,b}}^{\times}}(z)=\sum\limits_{A\subseteq E(B)}z^{\varepsilon(B^{\times|A})} + \sum\limits_{A\subseteq E(B)}z^{\varepsilon({{B}^{\prime}_{a,b}}^{\times|A})}=\sum\limits_{i=1}^{4}\sum\limits_{A\subseteq E_i} (z^{\varepsilon(B^{\times|A})}+z^{\varepsilon({B^{\prime}_{a,b}}^{\times|A})}).$$
\hfill{$\square$}
\subsection{A proof of Theorem \ref{rankoffourterm}}
Inspired by Theorem \ref{pPpva}, a modified polynomial of the partial Petrial polynomial is given.

\vspace{-0.3em}
\begin{defi}\label{modifiedpartialPetrial polynomial}
	The modified partial Petrial polynomial of any bouquet $B$ is defined as %the generating~function
	%\vspace{-0.3em}
	$$^{M\partial}{\varepsilon^{\times}_{B}}(z)=\sum\limits_{A\subseteq E(B)}(-1)^{|A|}z^{\varepsilon{(B^{\times|A})}}.$$
	%\noindent that enumerates partial Petrials of $B$ by Euler genus. 
\end{defi}

Lemma \ref{mfourterm} can be proved in a similar way as the proof of Theorem \ref{pPpva}. 
\vspace{-0.3em}
\begin{lem}\label{mfourterm}
	The modified partial Petrial polynomial satisfies the four-term relation.
\end{lem}

%We can observe that the modified partial Petrial polynomial satisfies the four-term relation. 

By Lemma \ref{rank=genusofbouquet}, the modified partial Petrial polynomial has the following equivalent expression. 

\vspace{-0.3em}
\begin{thm}\label{modifiedpartialPetrialpolynomial_rank}
	The {modified partial Petrial polynomial} of any bouquet $B$ has an exact expression: 
	$$^{M\partial}{\varepsilon^{\times}_{B}}(z)=\sum\limits_{A\subseteq E(B)}(-1)^{|A|}z^{\text{rank}(\text{adj}(SI(B^{\times|A})))}.$$
	%\noindent that enumerates partial Petrials of $B$ by Euler genus. 
\end{thm}

%\begin{thm}
%Let $B$ be a framed chord diagram and let $a$ and $b$ be two adjacent chords of $B$, then the modified partial Petrial polynomial satisfies the following four-term relation
%\[^{M\partial}{\varepsilon_{B}^{\times}}(z) -^{M\partial}{\varepsilon_{\widetilde{B}_{a,b}}^{\times}}(z)+ ^{M\partial}{\varepsilon_{\widetilde{B^{\prime}}_{a,b}}^{\times}}(z) -^{M\partial}{\varepsilon_{B^{\prime}_{ab}}^{\times}}(z)=0.
%\]
%\end{thm}

\vspace{-0.3em}
In \cite{Petrial_Yan}, Yan and Li showed that the partial Petrial polynomial of a bouquet is determined by its intersection graph. We show a similar property for the {modified partial Petrial polynomial}.

If two bouquets $B_1$ and $B_2$ have the same intersection graph, there exists a bijection, denoted by $\phi$, between $E(B_1)$ and $E(B_2)$. 
\begin{lem}
	For bouquets \( B_1 \) and \( B_2 \) with the same intersection graph, let $T_1$ $($resp. $T_2$$)$ be the set of all twisted edges of $B_1$ $($resp. $B_2$$)$ and $\phi$ be the bijection between $E(B_1)$ and $E(B_2)$.
	Then  
	\vspace{-0.3em}
	$$^{M\partial}{\varepsilon^{\times}_{B_1}}(z)=(-1)^{|\phi(T_1)\Delta T_2|}\cdot{}^{M\partial}{\varepsilon^{\times}_{B_2}}(z).$$ 
\end{lem}
\vspace{-1.2em}
\begin{proof}
	 Let $g$ be a mapping between the families of subsets of $E(B_1)$ and $E(B_2)$ defined by $g(A)=\phi(A)\Delta(\phi(T_1)\Delta T_2)$ for $A\subseteq E(B_1)$. Thus $g$ is a bijection and $SI(B_1^{\times|A})=SI(B_2^{\times|g(A)})$ by the definition of partial Petrial graph, which implies that  $\varepsilon(B_1^{\times| A})=\varepsilon(B_2^{\times|g(A)})$ by Lemma \ref{varepsilion_SIB}. Since $|\phi(A)\Delta (\phi(T_1)\Delta T_2)|=|\phi(A)|+|\phi(T_1)\Delta T_2|-2|\phi(A)\cap (\phi(T_1)\Delta T_2)|$, $|g(A)|$ and $|\phi(A)|+|\phi(T_1)\Delta T_2|$ are of the same parity. Thus $|g(A)|$ and $|\phi(A)|$ have the same parity if and only if $|\phi(T_1)\Delta T_2|$ is even. By definition \ref{poly_Petrial}, one has
	 \begin{align*}\label{equation_B=-B}
	 	^{M\partial}{\varepsilon^{\times}_{B_1}}(z)&=\sum\limits_{A\subseteq E(B_1)}(-1)^{|A|}z^{\varepsilon{(B_1^{\times|A})}}\\
	 	%&=\sum\limits_{A\subseteq E(B_1)}(-1)^{|A|}z^{\varepsilon{(B_2^{\times|g(A)})}}\\
	 	&=\sum\limits_{g(A)\subseteq E(B_2)}(-1)^{|\phi(A)|}z^{\varepsilon{(B_2^{\times|g(A)})}}\\
	 	&=\begin{cases}
	 		\sum\limits_{g(A)\subseteq E(B_2)}(-1)^{|g(A)|}z^{\varepsilon{(B_2^{\times|g(A)})}},
	 		& \text{if~} |\phi(T_1)\Delta T_2| \text{~is~even}\\
	 		\sum\limits_{g(A)\subseteq E(B_2)}(-1)^{|g(A)|+1}z^{\varepsilon{(B_2^{\times|g(A)})}},
	 		& \text{if~} |\phi(T_1)\Delta T_2| \text{~is~odd}\\
	 	\end{cases}\\
	 	%&=\begin{cases}
	 	%	^{M\partial}{\varepsilon^{\times}_{B_2}}(z),
	 	%	& \text{if~} |\phi(T_1)\Delta T_2| \text{~is~even}\\
	 	 %   -^{M\partial}{\varepsilon^{\times}_{B_2}}(z),
	 	%	& \text{if~} |\phi(T_1)\Delta T_2| \text{~is~odd}\\
	 	%\end{cases}\\
	 	&=(-1)^{|\phi(T_1)\Delta T_2|}\cdot{}^{M\partial}{\varepsilon^{\times}_{B_2}}(z).
	 \end{align*}	
	 
	 It is finished.
\end{proof}
\begin{coro}\label{signedintersectiondeterminedMpPp}
	Bouquets with the same signed intersection graph have the same modified partial Petrial polynomial.
\end{coro}

Inspired by Corollary \ref{signedintersectiondeterminedMpPp}, we can reduce the modified partial Petrial polynomial of bouquets to this polynomial of its signed intersection graphs. Intersection graphs are simple graphs, but not every simple graph is an intersection graph \cite{circle_graph}. Thus, we extend the definition of modified partial Petrial polynomial from signed intersection graphs (bouquets) to all simple signed graphs.%Since signed intersection graphs are signed simple graphs, we make the following generalization definition.

\begin{defi}\label{extendpPp}
For a simple signed graph $S$, let $X$ be the set of all vertices with a negative sign. The {modified partial Petrial polynomial} of $S$, denoted by $^{M\partial}{\varepsilon^{\times}_{S}}(z)$, is defined as %the generating function
\[^{M\partial}{\varepsilon^{\times}_{S}}(z)=
%\sum\limits_{A\subseteq V(S)}(-1)^{|A|}z^{{\text{rank}(\text{adj}(S)+\mathbb{D}_{X}+\mathbb{D}_A)}_{\mathbb{Z}_2}}=
\sum\limits_{A\subseteq V(S)} (-1)^{|A|} z^{\text{rank}(\text{adj}(S)+\mathbb{D}_{X\Delta A}) %_{\mathbb{Z}_2}
}.\]
	%\noindent that enumerates partial Petrials of $B$ by Euler genus. 
\end{defi}

In order to prove the modified partial Petrial polynomial is $4$-invariant, we define two graph operations $S^{\prime}_{a,b}$ and $\widetilde{S}_{a,b}$, on a simple signed graph $S$, which corresponds to non-orientable bouquets version if $S$ is a signed intersection graph. Consider %ing a simple signed graph $S$ and 
two distinct vertices $a, b\in V(S)$. %, there are two associated graphs $S_{a,b}^{\prime}$ and $\widetilde{S}_{a,b}$. 
The graph $S_{a,b}^{\prime}$ is obtained from $S$ by removing the edge $ab$ if it exists, or by adding it if it does not exist. Let $N(b)$ be the set of all neighbours of the vertex $b$ in $S$. The graph $\widetilde{S}_{a,b}$ is obtained from $S$ by first changing the adjacency between $a$ and all vertices of $N(b)$, and then changing the sign of $a$ if the sign of $b$ is negative; otherwise, the sign of $a$ is unchanged. \\
\indent By the definition of $S^{\prime}_{a,b}$  (resp. $\widetilde{S}_{a,b}$), Lemma \ref{adj(exchange)} (resp. \ref{adj(slide)}) holds, which shows the relationship between the adjacency matrices of $S$ and $S^{\prime}_{a,b}$  (resp. $\widetilde{S}_{a,b}$).

\begin{lem}\label{adj(exchange)}
	Let $S$ be a simple signed graph, then $\text{adj}(S^{\prime}_{a,b})$ is obtained from $\text{adj}(S)$ by only adding one to the two elements at the intersection of the row corresponding to vertex $a$ (resp. $b$) and the column corresponding to vertex $b$ (resp. $a$) in matrix $\text{adj}(S)$ over $\mathbb{GF}(2)$.
\end{lem}
\begin{lem}\label{adj(slide)}
	Let $S$ be a simple signed graph, then the adjacency matrix 
	$\text{adj}(\widetilde{S}_{a,b})$ is obtained from $\text{adj}(S)$ by adding the row and column corresponding to the vertex $b$ to the row and column corresponding to the vertex $a$ simultaneously, respectively, over $\mathbb{GF}(2)$ and keeping all elements in other rows and columns unchanged.
\end{lem}

% Because elementary transformations do not change the rank of a matrix, Lemma \ref{rankI=rankI_slide} holds.
 
%\begin{lem}\label{rankI=rankI_slide}
%For a simple signed graph $S$ and two vertices $a, b\in V(S)$, then $$\text{rank}(\text{adj}(S))_{\mathbb{Z}_2}=\text{rank}(\text{adj}(\widetilde{S}_{a,b}))_{\mathbb{Z}_2}$$
%\end{lem}
%\begin{thm}
%For any simple signed graph $S$ and any pair of vertices $a, b \in V(S)$, the extented partial Petrial polynomial $^{G\partial}{\varepsilon^{\times}_{S}}(z)$ of  $S$ satisfies the following four-term relation:
%\[^{G\partial}{\varepsilon_{S}^{\times}}(z) -^{G\partial}{\varepsilon_{\widetilde{S}_{a,b}}^{\times}}(z) + ^{G\partial}{\varepsilon_{\widetilde{S^{\prime}}_{a,b}}^{\times}}(z) - ^{G\partial}{\varepsilon_{S^{\prime}_{ab}}^{\times}}(z)=0.
%\]
%\end{thm}
Next, we prove Theorem \ref{rankoffourterm}, which implies that the modified partial Petrial polynomial of all simple signed graphs is a $4$-invariant. This provides an answer of Problem \ref{Problem_4-invariants}.

\vspace{0.6em}
\noindent\emph{\textbf{Proof of Theorem \emph{\ref{rankoffourterm}}}~~}
To prove that the {modified partial Petrial polynomial} of any simple signed graph is a $4$-invariant, by the definition of $4$-invariant, we only need to prove that 
%the modified partial Petrial polynomial for any  simple signed graph $S$ satisfies 
$^{M\partial}{\varepsilon_{S}^{\times}}(z) -^{M\partial}{\varepsilon_{\widetilde{S}_{a,b}}^{\times}}(z) + ^{M\partial}{\varepsilon_{\widetilde{S^{\prime}}_{a,b}}^{\times}}(z) - ^{M\partial}{\varepsilon_{S^{\prime}_{a,b}}^{\times}}(z)=0$
for any  simple signed graph $S$ and any pair of distinct vertices $a, b \in V(S)$. It suffices to prove that 
\begin{equation}\label{eqth4.1}
	^{M\partial}{\varepsilon_{S}^{\times}}(z)  + ^{M\partial}{\varepsilon_{\widetilde{S^{\prime}}_{a,b}}^{\times}}(z) =
	^{M\partial}{\varepsilon_{\widetilde{S}_{a,b}}^{\times}}(z)+^{M\partial}{\varepsilon_{S^{\prime}_{a,b}}^{\times}}(z).
\end{equation}
 
Assume that $X$ is the set of all vertices of $S$ with a negative sign. We divide the power set of $V(S)$ into the disjoint union of four
parts in the following two different ways:\\
\indent ($1$) $V(S)=V_1\cup V_2 \cup V_3 \cup V_4$, where
$V_1 = \{A \subseteq V(S)\mid a,b\notin A\Delta X\}$,
$V_2 = \{A \subseteq V(S)\mid a\in A\Delta X, b\notin A\Delta X\}$, $V_3 = \{A \subseteq V(S)\mid a\notin A\Delta X, b\in A\Delta X\}$ and $V_4 = \{A \subseteq V(S)\mid a,b\in A\Delta X\}$; \\
\indent ($2$) $V(S)=V_1^{\prime}\cup V_2^{\prime} \cup V_3^{\prime} \cup V_4^{\prime}$, where
$V_1^{\prime} = \{A \subseteq V(S)\mid a,b\notin A\}$,
$V_2^{\prime} = \{A \subseteq V(S)\mid a\in A, b\notin  A\}$, 
$V_3^{\prime} = \{A \subseteq V(S)\mid a\notin A, b\in A\}$ and 
$V_4^{\prime} = \{A \subseteq V(S)\mid a,b\in  A\}$.\\
\noindent We first prove that the following claim holds.

\vspace{0.3em}
\noindent\textbf{Claim.} There is a bijection $m: \{1,2,3,4\}\rightarrow\{1,2,3,4\}$ such that $X\Delta A\in V_i$ if and only if $A\in V_{m(i)}^{\prime}$ for any $A\subseteq V(S)$. 

\vspace{0.3em}
\noindent\emph{\textbf{Proof of Claim.}}
According to the definition of the four-term relation as shown in Fig. \ref{Fig_four-term}, we only need to discuss three cases as follows.

For $i\in\{1,2,3,4\}$, if $a,b\notin X$, let $m(i)=i$; if $a\in X,b\notin X$, let $m(i)=i+(-1)^{i+1}$; 
%\begin{align*}
%	m(i)=	\begin{cases}
	%	2, 
	%	& \text{if~} i=1;\\
	%	1,
	%	& \text{if~} i=2;\\
	%	4, 
	%	& \text{if~} i=3;\\
	%	3,
	%	& \text{if~} i=4;
	%\end{cases}
%\end{align*}
if $a\notin X,b\in X$, let $m(i)=(i+2)\pmod 4$.
  %\begin{align*}
%	m(i)=	\begin{cases}
	%	3, 		& \text{if~} i=1;\\
	%	4,      & \text{if~} i=2;\\
	%	1,   	& \text{if~} i=3;\\
	%	2,   	& \text{if~} i=4.
	%\end{cases}
%\end{align*}

For each case, we will check that $X\Delta A\in V_i$ if and only if $A\in V_{m(i)}^{\prime}$ for any $A\subseteq V(S)$. Since all verifications are similar, we only check the case that $a,b\notin X$. 
By the definitions of $V_i$ and $V_i^{\prime}$, $i\in \mathbb{Z}_4$, one has that
\begin{align*}
	\begin{cases}
		A\in V_1 \Leftrightarrow a,b\notin A\Delta X \Leftrightarrow a,b\notin A \Leftrightarrow A\in V_1^{\prime};\\
	    A\in V_2 \Leftrightarrow a\in A\Delta X,b\notin A\Delta X \Leftrightarrow a\in A,b\notin A \Leftrightarrow A\in V_2^{\prime};\\
	    A\in V_3 \Leftrightarrow a\notin A\Delta X,b\in A\Delta X \Leftrightarrow a\notin A,b\in A \Leftrightarrow A\in V_3^{\prime};\\
	    A\in V_4 \Leftrightarrow a,b\in A\Delta X \Leftrightarrow a,b\in A \Leftrightarrow A\in V_4^{\prime}.\\
	\end{cases}
\end{align*}
This claim is proved.

Next, we will prove Equation (\ref{eqth4.1}). One has that
\begin{equation}\label{eq_2}
	\begin{aligned}
		\!\!\!^{M\partial}{\varepsilon_{S}^{\times}}(z) \!+\!\! ^{M\partial}{\varepsilon_{\widetilde{S^{\prime}}_{a,b}}^{\times}}\!(z)
		\!&=\!\sum\limits_{A\subseteq V(S)}(-1)^{|A|}(z^{\text{rank}(\text{adj}(S)+\mathbb{D}_{X\Delta A})%_{\mathbb{Z}_2}
		} \!+\! z^{\text{rank}(\text{adj}(\widetilde{S^{\prime}}_{a,b})+\mathbb{D}_{X\Delta A})%_{\mathbb{Z}_2}
		})
		&   \text{Definition}~\ref{extendpPp}\\
		\!&=\!\sum\limits_{i=1}^4\!\sum\limits_{A\in V_i}\!(-1)^{|A|}(z^{\text{rank}(\text{adj}(S)+\mathbb{D}_{X\Delta A})%_{\mathbb{Z}_2}
		} \!+\! z^{\text{rank}(\text{adj}(\widetilde{S^{\prime}}_{a,b})+\mathbb{D}_{X\Delta A})%_{\mathbb{Z}_2}
		})\\
		\!&=\! \sum\limits_{i=1}^4\!\sum\limits_{Y\Delta X\in V_i}\!\!(-1)^{|Y\Delta X|}(z^{\text{rank}(\text{adj}(S)+\mathbb{D}_{Y})%_{\mathbb{Z}_2}
		} \!+\! z^{\text{rank}(\text{adj}(\widetilde{S}^{\prime}_{a,b})+\mathbb{D}_{Y})%_{\mathbb{Z}_2}
		}) &  Y \! =\! X\Delta A\\
		\!&=\! \sum\limits_{i=1}^4\!\sum\limits_{Y\in V_i^{\prime}}(-1)^{|Y\Delta X|}(z^{\text{rank}(\text{adj}(S)+\mathbb{D}_{Y})%_{\mathbb{Z}_2}
		} + z^{\text{rank}(\text{adj}(\widetilde{S}^{\prime}_{a,b})+\mathbb{D}_{Y})%_{\mathbb{Z}_2}
		}) &  \text{Claim}\\
	\end{aligned} 
\end{equation}
and by a similar way, we have
\begin{equation}\label{eq_3}
	\begin{aligned}
		^{M\partial}{\varepsilon_{\widetilde{S}_{a,b}}^{\times}}(z)  + ^{M\partial}{\varepsilon_{{S^{\prime}}_{a,b}}^{\times}}(z)
		=\sum\limits_{i=1}^4\sum\limits_{Y\in V_i^{\prime}}(-1)^{|Y\Delta X|}(z^{\text{rank}(\text{adj}(\widetilde{S}_{a,b})+\mathbb{D}_{Y})%_{\mathbb{Z}_2}
		} + z^{\text{rank}(\text{adj}({S}^{\prime}_{a,b})+\mathbb{D}_{Y})%_{\mathbb{Z}_2}
		}).\\
	\end{aligned} 
\end{equation}

If $Y\in V_1^{\prime}\cup V_2^{\prime}$, one has $b\notin Y$, which implies that the diagonal element of matrix $\mathbb{D}_Y$ at vertex $b$ is 0. Recall that $\text{adj}(\widetilde{S}_{a,b})$ is obtained from $\text{adj}(S)$ by adding the row and column of the vertex $b$ to the row and column of the vertex $a$ simultaneously, respectively, over $\mathbb{GF}(2)$ and keeping all elements in other rows and columns unchanged as Lemma \ref{adj(slide)}. Thus $\text{adj}(\widetilde{S}_{a,b})+\mathbb{D}_Y$ is obtained from $\text{adj}(S)+\mathbb{D}_Y$ with the same elementary transformations of matrix. 
Because elementary transformations do not change the rank of a matrix, we have that 
\begin{equation}\label{Eq4}
	\text{rank}(\text{adj}({S})+\mathbb{D}_Y)%_{\mathbb{Z}_2}
	=\text{rank}(\text{adj}(\widetilde{S}_{a,b})+\mathbb{D}_{Y})%_{\mathbb{Z}_2}
	~~~~~~~~~~~~~\text{for any $Y\in V_1^{\prime}\cup V_2^{\prime}$}
\end{equation} 
and 
\begin{equation}\label{Eq5}
	\text{rank}(\text{adj}({S^{\prime}})+\mathbb{D}_Y)%_{\mathbb{Z}_2}
	=\text{rank}(\text{adj}(\widetilde{S}^{\prime}_{a,b})+\mathbb{D}_{Y})%_{\mathbb{Z}_2}
	~~~~~~~~~~~~~\text{for any $Y\in V_1^{\prime}\cup V_2^{\prime}$}.
\end{equation} 

Assume that the block matrix representation of matrix $\text{adj}(S)$ is 
\[\text{adj}(S) = 
\begin{array}{c@{}c}
	& \begin{array}{ccc}
		V(S)\backslash \{a,b\} & ~~a~~~ & ~~~b~~~ \\
	\end{array} \\
	\begin{array}{c}
		V(S)\backslash\{a,b\} \\ a \\ b
	\end{array} &
	\left[
	\begin{array}{ccc}
		~~~~X_1~~~~~~ & ~~~X_2~~~ & ~~X_3 \\
		~~~~X_2'~~~~~~ & ~~~x_4~~~ & ~~x_5 \\
		~~~~X_3'~~~~~~ & ~~~x_5~~~ & ~~x_6 \\
	\end{array}
	\right].
\end{array}
\]
By Lemma \ref{adj(exchange)} and Lemma \ref{adj(slide)}, we have that 
\[
\begin{array}{cc}
	\text{adj}(\widetilde{S}_{a,b})= \begin{bmatrix}
		X_1 & X_2+X_3 & X_3 \\
		X_2^{\prime}+X_3^{\prime} & x_4+x_6 & x_5+x_6 \\
		X_3^{\prime} & x_5+x_6 & x_6 \\
	\end{bmatrix},
	&
	\text{adj}(S^{\prime}_{a,b})= \begin{bmatrix}
		X_1 & X_2 & X_3 \\
		X_2^{\prime} & x_4 & x_5+1 \\
		X_3^{\prime} & x_5+1 & x_6 \\
	\end{bmatrix}
\end{array}
\]
and 
\[\text{adj}(\widetilde{S}^{\prime}_{a,b})= \begin{bmatrix}
	X_1 & X_2+X_3 & X_3 \\
	X_2^{\prime}+X_3^{\prime} & x_4+x_6 & x_5+x_6+1 \\
	X_3^{\prime} & x_5+x_6+1 & x_6 \\
\end{bmatrix}.\]

If $Y\in V_3^{\prime}\cup V_4^{\prime}$, one has $b\in Y$. Assume that $d_1=1$ if $a\in Y$; otherwise, $d_1=0$. We have that
\[\text{adj}({S})+\mathbb{D}_Y= \begin{bmatrix}
		X_1+\mathbb{D}_{V(S)\backslash\{a,b\}} & X_2 & X_3 \\
		X_2^{\prime} & x_4+d_1 & x_5 \\
		X_3^{\prime} & x_5 & x_6+1 \\
	\end{bmatrix}\]
and
\[\text{adj}(\widetilde{S}^{\prime}_{a,b}+\mathbb{D}_{Y\Delta\{a\}})= \begin{bmatrix}
	X_1+\mathbb{D}_{V(S)\backslash\{a,b\}}  & X_2+X_3 & X_3 \\
	X_2^{\prime}+X_3^{\prime} & x_4+x_6+d_1+1 & x_5+x_6+1 \\
	X_3^{\prime} & x_5+x_6+1 & x_6+1 \\
\end{bmatrix}.\]
It follows that the matrix $\text{adj}({S})+\mathbb{D}_Y$ is obtained from
$\text{adj}(\widetilde{S}^{\prime}_{a,b})+\mathbb{D}_{Y\Delta\{a\}}$ by adding 
the row and column of the vertex $b$ to the row and column of the vertex $a$ simultaneously over $\mathbb{GF}(2)$. Thus%, which implies that 
\begin{equation}\label{Eq6}
	%\begin{aligned}
	\text{rank}(\text{adj}({S})+\mathbb{D}_Y)%_{\mathbb{Z}_2}
	=\text{rank}(\text{adj}(\widetilde{S}^{\prime}_{a,b})+\mathbb{D}_{Y\Delta \{a\}})%_{\mathbb{Z}_2}
	~~~~~~~~\text{for any $Y\in V_3^{\prime}\cup V_4^{\prime}$}.
	%\end{aligned}
\end{equation}
In a similar way, one has that
\begin{equation}\label{Eq7}
	%\begin{aligned}
	\text{rank}(\text{adj}(\widetilde{S}_{a,b})+\mathbb{D}_Y)%_{\mathbb{Z}_2}
	=\text{rank}(\text{adj}({S}^{\prime}_{a,b})+\mathbb{D}_{Y\Delta \{a\}})%_{\mathbb{Z}_2}
	~~~~~~~~\text{for any $Y\in V_3^{\prime}\cup V_4^{\prime}$}.
	%\end{aligned}
\end{equation}
\noindent By Equations (\ref{eq_2}),(\ref{Eq4}),(\ref{Eq5}) and (\ref{Eq6}), one has that
\begin{equation}\label{eq_8}
	\begin{aligned}
		\!^{M\partial}{\varepsilon_{S}^{\times}}(z) \!+\!\! ^{M\partial}{\varepsilon_{\widetilde{S^{\prime}}_{a,b}}^{\times}}\!(z)
		\!&=\!\sum\limits_{i=1}^2 \! \sum\limits_{Y\in V_i^{\prime}}(-1)^{|Y\Delta X|}(z^{\text{rank}(\text{adj}(\widetilde{S}_{a,b})+\mathbb{D}_{Y})}\!+\!z^{\text{rank}(\text{adj}({S}^{\prime}_{a,b})+\mathbb{D}_{Y})})\\
		\!&+\!
		\sum\limits_{i=3}^4 \! \sum\limits_{Y\in V_i^{\prime}}(-1)^{|Y\Delta X|}(z^{\text{rank}(\text{adj}(\widetilde{S}^{\prime}_{a,b})+\mathbb{D}_{Y\Delta \{a\}})} \!+\! z^{\text{rank}(\text{adj}(\widetilde{S}^{\prime}_{a,b})+\mathbb{D}_{Y})})  \\
    	\! &=\! \sum\limits_{i=1}^2 \! \sum\limits_{Y\in V_i^{\prime}}(-1)^{|Y\Delta X|}(z^{\text{rank}(\text{adj}(\widetilde{S}_{a,b})+\mathbb{D}_{Y})}+z^{\text{rank}(\text{adj}({S}^{\prime}_{a,b})+\mathbb{D}_{Y})}).\\
	\end{aligned} 
\end{equation}
The second equation follows from the fact that the sum $\sum\limits_{Y\in V_4^{\prime}}(-1)^{|Y\Delta X|}(z^{\text{rank}(\text{adj}(\widetilde{S}^{\prime}_{a,b})+\mathbb{D}_{Y\Delta \{a\}})}+ z^{\text{rank}(\text{adj}(\widetilde{S}^{\prime}_{a,b})+\mathbb{D}_{Y})})$ equals $\sum\limits_{Y\in V_3^{\prime}}\!(-1)^{|Y\Delta\{a\}\Delta X|}(z^{\text{rank}(\text{adj}(\widetilde{S}^{\prime}_{a,b})+\mathbb{D}_{Y})}  +  z^{\text{rank}(\text{adj}(\widetilde{S}^{\prime}_{a,b})+\mathbb{D}_{Y\Delta \{a\}})}).$ 
%$\sum\limits_{Y\in V_4^{\prime}}(-1)^{|Y\Delta X|}(z^{\text{rank}(\text{adj}(\widetilde{S}^{\prime}_{a,b})+\mathbb{D}_{Y\Delta \{a\}})}+ z^{\text{rank}(\text{adj}(\widetilde{S}^{\prime}_{a,b})+\mathbb{D}_{Y})}) = \sum\limits_{Y\in V_3^{\prime}}\!(-1)^{|Y\Delta\{a\}\Delta X|}(z^{\text{rank}(\text{adj}(\widetilde{S}^{\prime}_{a,b})+\mathbb{D}_{Y})}  +  z^{\text{rank}(\text{adj}(\widetilde{S}^{\prime}_{a,b})+\mathbb{D}_{Y\Delta \{a\}})}).$ 
In a similarly way, by Equations (\ref{eq_3}) and (\ref{Eq7}), we have that
\begin{equation}\label{eq_9}
	\begin{aligned}
		^{M\partial}{\varepsilon_{\widetilde{S}_{a,b}}^{\times}}(z)  + ^{M\partial}{\varepsilon_{{S^{\prime}}_{a,b}}^{\times}}(z)
		&=\sum\limits_{i=1}^2\sum\limits_{Y\in V_i^{\prime}}(-1)^{|Y\Delta X|}(z^{\text{rank}(\text{adj}(\widetilde{S}_{a,b})+\mathbb{D}_{Y})}+z^{\text{rank}(\text{adj}({S}^{\prime}_{a,b})+\mathbb{D}_{Y})}) \\
		&+ \sum\limits_{i=3}^4\sum\limits_{Y\in V_i^{\prime}}(-1)^{|Y\Delta X|}(z^{\text{rank}(\text{adj}({S}^{\prime}_{a,b})+\mathbb{D}_{Y\Delta \{a\}})}+z^{\text{rank}(\text{adj}({S}^{\prime}_{a,b})+\mathbb{D}_{Y})})\\
		&=\sum\limits_{i=1}^2\sum\limits_{Y\in V_i^{\prime}}(-1)^{|Y\Delta X|}(z^{\text{rank}(\text{adj}(\widetilde{S}_{a,b})+\mathbb{D}_{Y})}+z^{\text{rank}(\text{adj}({S}^{\prime}_{a,b})+\mathbb{D}_{Y})}).\\
	\end{aligned} 
\end{equation}
Thus $^{M\partial}{\varepsilon_{S}^{\times}}(z) +
^{M\partial}{\varepsilon_{\widetilde{S}^{\prime}_{a,b}}^{\times}}(z)  =^{M\partial}{\varepsilon_{\widetilde{S}_{a,b}}^{\times}}(z)  + ^{M\partial}{\varepsilon_{S^{\prime}_{a,b}}^{\times}}(z)$ by Equations (\ref{eq_8}) and (\ref{eq_9}). 

It is finished.
\hfill{$\square$}

\section{Conclusion}
In this paper, we prove that the Euler genus of a ribbon graph $G$ equals the rank of an adjacency matrix associated with $G$ and
give an equivalent rank-based expression for the partial Petrial polynomial of a ribbon graph, which transforms problems in topological graph theory into an algebraic framework. 
Furthermore, we 
derive an equivalent relation on four chord diagrams in terms of their partial Petrial polynomials and extend the partial Petrial polynomial by assigning coefficients +1 or -1  to its terms in the partial Petrial polynomial, which satisfies the four-term relation of graphs.  %In the future, we propose the following natural problem:

%\begin{Problem}
%Can we find an algorithm that, based on its intersection graph, determines a Euler-genus-minimizing partial Petrial graph of $B_n$\emph{?}
%Under what necessary and sufficient conditions does the number of terms in the Partial Petrial polynomial of ribbon graph remain equal between a ribbon graph and its subgraph\emph{?}
%\end{Problem}

%the genus saturates as trees are incorporated
%In this paper, we prove that the partial-dual Euler-genus polynomial $^{\partial}{\varepsilon_{B}}(z)$ of a prime bouquet $B$ is odd-interpolation if it contains a first-order term, which gives an answer  of Problem \ref{problem} for the odd case. In fact, there exist prime bouquets whose polynomials contain a first-order term but lack the second-order term. For instance, let $B=(1,2,4,3,-1,5,-2,-3,5,4)$, then  $^{\partial}{\varepsilon_{B}}(z)=12z^5 + 8z^4 + 10z^3 + 2z$, which implies that $^{\partial}{\varepsilon_{B}}(z)$ is both odd-interpolating and even-interpolating, but is not interpolating. %The characterization of a prime bouquet $B$ such that $^{\partial}{\varepsilon_{B}}(z)$ contains a second-order term is given in Theorem \ref{mainthm_second}. 

\section*{Acknowledgment}

\indent
This work was partially supported by the National Natural Science Foundation of China
(Nos. 12471321 and 12331013).


\begin{thebibliography}{99}
\bibitem{Beck}
I. Beck, Cycle decomposition by transpositions,  \textit{Journal of Combinatorial Theory, Series B}, 23 (2) (1977): 198–207.

\vspace{-0.6em}
\bibitem{2002_Ribbon graph}
B. Bollob\'{a}s and O. Riordan, A polynomial of graphs on surfaces, \textit{Mathematische Annalen}, 323 (2002): 81-96.

\vspace{-0.6em}
\bibitem{fourterm_Chmutov}
S. Chmutov, Partial-dual genus polynomial as a weight system, \textit{Communications in Mathematics}, 31 (2023): 113-124.

\vspace{-0.6em}
\bibitem{circle_graph}
S. Chmutov, S. Duzhin and J. Mostovoy, \textit{Introduction to Vassiliev knot invariants}, Cambridge University Press, 2012.

\vspace{-0.6em}
\bibitem{Counterexample_ChenYichao}
Q. Chen and Y. Chen, Parallel edges in ribbon graphs and interpolating behavior of partial-duality polynomials, \textit{European Journal of Combinatorics}, 102 (2022): 103492.

\vspace{-0.6em}
\bibitem{fourterm_chengzhiyun}
Z. Cheng, Partial-dual genus polynomial of graphs, \textit{European Journal of Combinatorics}, 130 (2025): 104221.

\vspace{-0.6em}
\bibitem{fourterm_DengYan}
Q. Deng, F. Dong, X. Jin and Q. Yan, Partial-dual polynomial as a framed weight system, \textit{Communications in Mathematics}, 31 (2023): 151-160.

\vspace{-0.6em}
\bibitem{2020_I}
J. Gross, T. Mansour and T. Tucker, Partial duality for ribbon graphs, I: Distributions, \textit{European Journal of Combinatorics}, 86 (2020): 103084.

\vspace{-0.6em}
\bibitem{2021_II}
J. Gross, T. Mansour and T. Tucker, Partial duality for ribbon graphs, II: Partial-duality polynomials and monodromy computations, \textit{European Journal of Combinatorics}, 95 (2021): 103329.

\vspace{-0.6em}
\bibitem{deffourterm1}
D. Ilyutko and V. Manturov, A parity map of framed chord diagrams, \textit{Journal of Knot Theory and Its Ramifications}, 24(13) (2015): 1541006.

\vspace{-0.6em}
\bibitem{Topology of Surfaces}
L. Kinsey, Topology of Surfaces, \textit{Springer Science and Business Media}, New York, 1993.

\vspace{-0.6em}
\bibitem{def2fourterm1}
S. Lando. J-invariants of plane curves and framed chord diagrams, \textit{Functional Analysis and Its Applications}, 40(1) (2006): 1–10.

\vspace{-0.6em}
\bibitem{problemfourtermrelation}
S. Lando, On a Hopf algebra in graph theory, \textit{Journal of  Combinatorial Theory, Series B}, 80 (1) (2000): 104–121.



%\vspace{-0.6em}
%\bibitem{matroid1}
%A. Bouchet, Greedy algorithm and symmetric matroids, \textit{Mathematical Programming}, 38 (1987): 147--159.

%\vspace{-0.6em}
%\bibitem{matroid2}
%C. Chun, I. Moffatt, S. Noble and R. Rueckriemen, Matroids, delta-matroids and embedded graphs, \textit{Journal of Combinatorial Theory, Series A}, 167 (2019): 7--59.

%\vspace{-0.6em}
%\bibitem{west}
%D. West, \textit{Introduction to graph theory}, Upper Saddle River: Prentice Hall, 2001.

%\vspace{-0.6em}
%\bibitem{Partial}
%S. Chmutov, Generalized duality for graphs on surfaces and the signed Bollob\'{a}s-Riordan polynomial, \textit{Journal of Combinatorial Theory, Series B}, 99(3) (2009): 617--638.

%\vspace{-0.6em}
%\bibitem{dualities, polynomials, and knots}
%J. Ellis-Monaghan and I. Moffatt, \textit{Graphs on surfaces: dualities, polynomials, and knots}, New York: Springer, 2013.

\vspace{-0.6em}
\bibitem{weight}
B. Mellor, A few weight systems arising from intersection graphs, \textit{Michigan Mathematical Journal}, 51(3) (2003): 509-536.

\vspace{-0.6em}
\bibitem{Join_Moffatt} 
I. Moffatt, Separability and the genus of a partial dual, \textit{European Journal of Combinatorics}, 34 (2013): 355-378.

\vspace{-0.6em}
\bibitem{2001_Graphs_on_surfaces}
B. Mohar and C. Thomassen, \textit{Graphs on surfaces}, Johns Hopkins University Press, 2001.

\vspace{-0.6em}
\bibitem{Introduction_Petrial}
S. Wilson, Operators over regular maps, \textit{The Pacific Journal of Mathematics}, 81 (1979): 559--568.


%\vspace{-0.6em}
%\bibitem{Yangyan}
%Y. Yang and X. Zha, Partial-dual Euler-genus distributions for bouquets with small Euler genus, \textit{Ars Mathematica Contemporanea}, 22 (2022): \#P3.09.




%\vspace{-0.6em}
%\bibitem{2007_SI_G}
%S. Chmutov and S. Lando, Mutant knots and intersection graphs, \textit{Algebraic and Geometric Topology}, 7(3) (2007): 1579--1598.

\vspace{-0.6em}
\bibitem{Yanqi_Intersection}
Q. Yan and X. Jin, Partial-dual polynomials and signed intersection graphs, \textit{Forum of Mathematics, Sigma}, 10 (2022): e69.

\vspace{-0.6em}
\bibitem{Matroids_Yan} 
Q. Yan and X. Jin, Partial-twuality polynomials of delta-matroids, \textit{Advances in Applied Mathematics}, 153 (2024): 102623.

\vspace{-0.6em}
\bibitem{Petrial_Yan}
Q. Yan and Y. Li, Partial Petrial polynomials for complete graphs and paths, \textit{Discrete Applied Mathematics}, 375 (2025): 281--289.



%\vspace{-0.6em}
%\bibitem{binomial_Yan}
%R. Feng, Q. Yan and X. Zheng, Characterizing circle graphs with binomial partial Petrial polynomials, \textit{arXiv preprint} (2025): arXiv:2507.02421.

\end{thebibliography}
\end{document}